\documentclass[article,12pt]{amsart}

\usepackage{amsfonts,url}
\usepackage{color}
\usepackage[margin=23mm]{geometry}
\usepackage{amssymb, amsmath, amsthm, amscd}
\usepackage[T1]{fontenc}
\usepackage{eucal,mathrsfs,dsfont}

\def\dist{\mathop{\rm dist}\nolimits}

\newcommand{\R}{\mathbb{R}}
\newcommand{\Rd}{ \mathbb{R}^{d}}
\newcommand{\Rdwithouto}{\mathbb{R}^{d}\setminus\{0\}}
\newcommand{\N}{\mathbb{N}}
\newcommand{\Z}{\mathbb{Z}}

\newcommand{\indyk}[1]{\mathds{1}_{#1}}
\newcommand{\sfera}{ \mathds{S}}
\newcommand{\Borel}{ {\mathcal{B}}(\Rd) }
\newcommand{\Fourier}{ {\mathcal{F}}}

\newcommand{\nubounded}[1]{\bar{\nu}_{#1}}

\newcommand{\scalp}[2]{#1\cdot#2}

\newcommand{\lin}{{\rm lin}}

 \def\dist{\mathop{\rm
    dist}\nolimits} \def\diam{\mathop{\rm diam}\nolimits}

\newtheorem{lemat}{Lemma}
\newtheorem{prop}{Proposition}
\newtheorem{twierdzenie}{Theorem}
\newtheorem{wniosek}[lemat]{Corollary}

\newtheorem{uwaga}{Remark}
\newtheorem{example}{Example}

\newcounter{conum} 

\renewcommand{\Re}{\ensuremath{\operatorname{Re}}}

\begin{document}

\title{Small time sharp bounds for kernels of convolution semigroups}
\author{Kamil Kaleta and Pawe{\l} Sztonyk}

\address{Kamil Kaleta, Institute of Mathematics,
  University of Warsaw,
  ul. Banacha 2,
  02-097 Warszawa
  and Institute of Mathematics and Computer Science \\ Wroc{\l}aw University of Technology
\\ Wyb. Wyspia{\'n}skiego 27, 50-370 Wroc{\l}aw, Poland}
\email{kkaleta@mimuw.edu.pl, kamil.kaleta@pwr.wroc.pl}

\address{Pawe{\l} Sztonyk \\ Institute of Mathematics and Computer Science,
  Wroc{\l}aw University of Technology,
  Wybrze{\.z}e Wyspia{\'n}\-skie\-go 27,
  50-370 Wroc{\l}aw, Poland}
\email{pawel.sztonyk@pwr.wroc.pl}

\maketitle

\begin{abstract}
We study small time bounds for transition densities of convolution semigroups corresponding to pure jump L\'evy processes in $\R^d$, $d \geq 1$, including those with jumping kernels exponentially and subexponentially localized at infinity. For a large class of L\'evy measures, non-necessarily symmetric {nor} absolutely continuous with respect to the Lebesgue measure, we find the optimal{, both in time and space,} upper bound for the corresponding transition kernels at infinity. In case of L\'evy measures that are symmetric and absolutely continuous, with densities $g$ such that $g(x) \asymp f(|x|)$ for nonincreasing profile functions $f$, we also prove the full characterization of the sharp two-sided transition densities bounds of the form
$$
p_t(x) \asymp h(t)^{-d} \cdot \indyk{\left\{|x|\leq \theta h(t)\right\}} + t \, g(x) \cdot \indyk{\left\{|x| \geq \theta h(t)\right\}}, \quad t \in (0,t_0), \ \ t_0>0, \ \ x \in \R^d.
$$
This is done for small and large $x$ separately. Mainly, our argument is based on new precise upper bounds for convolutions of L\'evy measures. Our investigations lead to {an} interesting and surprising dichotomy of the decay properties at infinity for transition kernels of pure jump L\'evy processes. All results are obtained  {solely} by analytic methods, {without} use of probabilistic arguments.


\bigskip
\noindent
2010 {\it MS Classification}: Primary 60G51, 60E07; Secondary 60J35, 47D03, 60J45 .\\
  
\end{abstract}

\footnotetext{
K. Kaleta was supported by the National Science Center (Poland) post-doctoral internship
grant on the basis of the decision No. DEC-2012/04/S/ST1/00093.

P. Sztonyk was supported by the National Science Center (Poland) grant on the basis of the decision No. DEC-2012/07/B/ST1/03356.
}
\section{Introduction and statement of results}

We study a convolution semigroup of probability measures $\{ P_t,\, t\geq 0 \}$ on $\R^d$, $d\in \left\{1,2,...\right\}$, determined by their Fourier transforms $\Fourier(P_t)(\xi)=\int_{\Rd} e^{i\scalp{\xi}{y}}P_t(dy)=\exp(-t\Phi(\xi))$, $t>0$, with the L\'evy-Khintchine exponent of the form
\begin{equation*}
  \Phi(\xi) =   \int_{\Rdwithouto} \left(1-e^{i\scalp{\xi}{y}}+i\scalp{\xi}{y}\indyk{B(0,1)}(y)\right)\nu(dy) - i\scalp{\xi}{b}, \quad \xi\in\Rd,
\end{equation*}
where $\nu$ is an infinite L\'evy measure on $\Rdwithouto$, i.e., $\int_{\Rdwithouto} \left(1\wedge |y|^2\right)\,\nu(dy) < \infty$ and \linebreak $\nu(\Rdwithouto)=\infty$, and $b\in\Rd$ is a drift term \cite{J1}. It is well known that there exists a pure jump L\'evy process $\{X_t,\,t\geq 0\}$ in $\R^d$ with transition {probabilities} given by $\{P_t,\,t\geq 0\}$ \cite{Sato} (in this paper, by pure jump L\'evy process we mean a process with no Gaussian component). The densities of {the} measures $P_t$ with respect to the Lebesgue measure are denoted by $p_t$, whenever they exist. For some sufficient and necessary conditions on the existence of transition densities $p_t$ we refer the reader to \cite{KSch}.

The function $p_t(x)$ is the fundamental solution ({the} heat kernel) of an evolution equation involving the infnitesimal generator
of the process $\{X_t,\,t\geq 0\}$, whose explicit expression is typically impossible to get. Therefore, it is a basic problem, both in probability theory and in analysis, to obtain sharp estimates of $p_t(x)$. In case of symmetric diffusions on $\R^d$, whose infinitesimal generators are uniformly elliptic and bounded divergence form operators, it is well known that the heat kernels enjoy the celebrated Aronson's Gaussian type estimates \cite{A}. 

The problem of estimates of transition densities for jump L\'evy processes has been intensively studied for many decades, mostly for stable processes \cite{BG60, PT69, H94, H03, G93, GH93, D91, W07, BS2007}. The general method of estimating {the} kernels of L\'evy semigroups is based on their convolutional structure and construction. Recent papers \cite{S10, S11, KnopKul, KSch1, KSz1, Knop13} contain the estimates for more general classes of L\'evy processes, including tempered processes with intensities of jumps lighter than polynomial. The paper \cite{BGR13} focuses on the estimates of densities for isotropic unimodal L\'evy processes with L\'evy-Khintchine exponents having the weak local
scaling at infinity, while the papers \cite{Mimica1, KSz1} discuss the processes with higher intensity of small jumps, remarkably different than {the} stable one. In \cite{ChKimKum, ChKum08, KSz} the authors investigate the case of more general, non-necessarily space homogeneous, symmetric jump Markov processes with jump intensities dominated by those of isotropic stable processes. Estimates of kernels for processes which are solutions of SDE{s} driven by L\'evy processes were obtained in \cite{Lean, Ishi94, Picard97, Pic, Ishi01}. For estimates of derivatives of L\'evy densities we refer the reader to \cite{S10a, BJ07, SchSW12, KSz1, KR13, Knop13}. In \cite{JKLSch2012} the authors gave a very interesting geometric interpretation of the transition densities for symmetric L\'evy processes. 

In the present paper, we focus on {a} special type of the small time behaviour of the densities $p_t$. Before we state our main results, we first need to introduce some necessary auxiliary notation and set the framework for our study. Denote
$$
  \Psi(r)=\sup_{|\xi|\leq r} \Re \Phi(\xi),\quad r>0.
$$
We note that $\Psi$ is continuous and non-decreasing{, and we also have} $\sup_{r>0} \Psi(r)=\infty$, since $\nu(\Rdwithouto)=\infty$.
Let 
$$\Psi^{-1}(s)=\sup\{r>0: \Psi(r)=s\} \quad \text{for} \ s\in (0,\infty)$$ 
so that $\Psi(\Psi^{-1}(s))=s$ for $s\in (0,\infty)$ and $\Psi^{-1}(\Psi(s))\geq s$ for $s>0$. To shorten the notation below, we set 
\begin{equation}\label{eq:def_br}
h(t):= \frac{1}{\Psi^{-1}\left(\frac{1}{t}\right)}
\quad  \text{and} \quad 
  b_r := \left\{
  \begin{array}{ccc}
    b - \int_{r<|y|<1} y \, \nu(dy) & \mbox{  if  } & r \leq 1,\\
    b + \int_{1<|y|<r} y \, \nu(dy) & \mbox{  if  } & r > 1.
  \end{array}\right. 
\end{equation}

{Substantial} part of our work {is concerned with} a large class of L\'evy measures that are non-necessarily symmetric {nor} absolutely continuous with respect to the Lebesgue measure. However, the sharpness of our results is most evident under the reasonable assumption that the L\'evy measure $\nu$ has the density $g(x)=g(-x)$ such that $g(x) \asymp f(|x|)$, $x \in \Rdwithouto$, for some nonincreasing function $f:(0,\infty) \to (0,\infty)$. With{in} this framework, we consider the following type of {small} time estimates of the densities $p_t(x)$, which is known to hold for a wide class of jump L\'evy processes (for simplicity we assume here that $b=0$). There are constants $c_1, C_1, C_2 \in (0,1]$, $ c_2, C_3, C_4 \geq 1$, $\theta> 0$ and $t_0>0$ such that
\begin{align} \label{eq:est1}
  C_1 \, [h(t)]^{-d} \leq p_t(x) \leq  C_3 \, [h(t)]^{-d}, \quad t \in (0,t_0], \quad |x| \leq \theta h(t),
\end{align}
and
\begin{align} \label{eq:est2}
  C_2 \, t \, f(c_2|x|) \leq p_t(x) \leq  C_4 \, t \,  f(c_1|x|), \quad t \in (0,t_0], \quad \theta h(t) \leq |x|,
\end{align}
where $f$ is the profile of the density $g$ of $\nu$. These two-sided bounds are called \emph{sharp} when $c_1 = c_2 = 1$ (the optimality of constants $C_1, C_2, C_3$ and $C_4$ is not required). Clearly, this is irrelevant when the profile function has {doubling} property, e.g. when $f(r)=r^{-d-\beta}$. However, if the {decay of the L\'evy measure at infinity is} faster than polynomial (e.g. $f(r) \asymp e^{-r}$ or $f(r) \asymp r^{-d-\beta} e^{-r}$ as $r \to \infty$), then it is {basic to establish where} sharp bounds hold or how sharp they {are}. Studying of sharpness and optimality of such results is one of the most difficult and fundamental problems in the modern theory of convolution semigroups. Many of {the results available} imply {bounds} for transition densities of jump processes with highly tempered L\'evy measures (see e.g. \cite{BGR13, ChKimKum, ChKum08, S10, S11, KSch1, KSz1, Knop13}). However, most of them are not sharp in the above sense. There is no comprehensive argument, or result, {which} ultimately explains and settles when exactly the sharp small time bounds for densities of jump L\'evy processes are satisfied. For many {examples} of tempered processes {with jump intensities exponentially localized at infinity} the problem of sharp estimates is still open and the sharpest possible estimates are not known. Such processes are important in mathematical physics (see \cite{CMS, KL}) and in financial mathematics (see e.g. \cite{BRKF} and references therein). In the present paper we address the problem of sharp small time bounds of integral kernels for a large class of convolution semigroups related to pure jump L\'evy processes. 

Beside some degenerate examples, in general, the L\'evy measures satisfy a kind of the doubling condition around zero. This property is inherited by the profile function and therefore we often have $f(c|x|) \asymp f(|x|)$ for any fixed $c>0$ and {all} small $x$. In many cases, the small time bounds of the densities $p_t(x)$ for small $x$ can be derived from the properties of the corresponding L\'evy-Khintchine exponent. Indeed, very often, by the Fourier transform, the asymptotics of $\Phi$ at infinity directly translates into the asymptotics of $p_t$ and $\nu$ at zero (see e.g. \cite{BGR13}). For large $x$ this picture is usually dramatically different. As we will see below, in this case the asymptotic behaviour of $p_t(x)$ strongly depends on {subtle} convolutional properties of the corresponding L\'evy measures. If the tail of the L\'evy measure {at infinity} is lighter than polynomial, then we can expect that $\lim_{r \to \infty} f(cr)/f(r) = \infty$, for all $c \in (0,1)$. In this case, the L\'evy-Khintchine exponent vanishes at zero quadratically and the sharp bounds of $p_t(x)$ for large $x$ and small $t$ cannot be derived from it. Furthermore, if we have the upper bound in \eqref{eq:est2} with $f(c_1|x|)$, for some $c_1 \in (0,1)$, then $f(c_1|x|)$ cannot be directly replaced by $cf(|x|)$ for any constant $c$. Of course, this does not mean that in this case the bound with the best possible rate $f(|x|)$ cannot hold. Unfortunately, in most cases it is too difficult to settle whether the worse rate in \eqref{eq:est2} is only a consequence of the flaw of the method or assumptions, or whether, perhaps, the bound of the form \eqref{eq:est2} with the exact rate $f(|x|)$ does not hold for large $x$. It is known that the L\'evy measure $\nu(dx)=g(x)\, dx$ is a vague limit of measures $P_t(dx)/t=(p_t(x)/t)\, dx$ as $t \to 0^{+}$ outside the origin, which may cause the false intuition that for small $t$ {both} functions $p_t(x)$ and $tg(x)$ should share exactly the same asymptotic properties. As we will show below, although sometimes sharp bounds in \eqref{eq:est2} seem to be possible or even evident, they surprisingly do not hold in general. Therefore it is quite reasonable to ask when exactly these bounds are satisfied in the{ir} sharpest form. 

The following Theorem \ref{th:main1} definitively resolves this problem for convolution semigroups built on symmetric L\'evy measures that are absolutely continuous with respect to the Lebesgue measure, with densities comparable to nonincreasing profiles. It gives {a} full characterization of sharp bounds \eqref{eq:est1}-\eqref{eq:est2} with {exact} rate $f(|x|)$ for all $x \in \R^d$. For later use we denote $g_r(\cdot)= g(\cdot) \indyk{B(0,r)^c} (\cdot)$, $r>0$.

\begin{twierdzenie}\label{th:main1} Let
$\nu(dx)=g(x)dx$ be a L\'evy measure such that $\nu(\Rdwithouto)=\infty$, $g(y)=g(-y)$ and {$g(x) \asymp f(|x|)$, $x \in \Rdwithouto$, for some nonincreasing function $f:\:(0,\infty)\to (0,\infty)$.} Moreover assume that $b \in \R^d$.

Then the following two conditions \textsc{(1.1)} and \textsc{(1.2)} are equivalent.
\begin{itemize}
\item[\textsc{(1.1)}] There exist $r_0>0$ and constants $L_1, L_2 >0$ such that the following two estimates 
\begin{itemize}
\item[(a)] $ g_{r_0} * g_{r_0}(x) \leq L_1 g(x)$, $|x| \geq 2 r_0$,
\item[(b)] $\Psi(1/|x|) \leq L_2 |x|^d g(x)$, $0<|x| \leq 2 r_0$,
\end{itemize}
hold.
\item[\textsc{(1.2)}] There exist $t_0, \theta >0$ and {constants} $C_{1} - C_{4}$ such that for every $t \in (0,t_0]$ the transition densities $p_t$ exist and satisfy
\begin{align*}
  C_{1} \, [h(t)]^{-d} \leq p_t(x+tb) \leq  C_{3} \, [h(t)]^{-d}, \quad t \in (0,t_0], \quad |x| \leq \theta h(t),
\end{align*}
and
\begin{align*}
  C_{2} \, t \, g(x) \leq p_t(x+tb) \leq  C_{4} \, t \, g(x), \quad t \in (0,t_0], \quad  |x| \geq \theta h(t).
\end{align*}
\end{itemize}
\end{twierdzenie}
\noindent
Both conditions (a) and (b) in \textsc{(1.1)} are local {in} the sense that they refer to {distinct} ranges of $x$. The proof of Theorem \ref{th:main1}, which consists of {two separate} parts, for small and large $x$, also reflects this property. In particular, the next result states that the condition \textsc{(1.1)} (b) in fact characterizes the bounds \textsc{(1.2)} for small $x$. It is worth to point out that under our assumptions the estimate opposite to \textsc{(1.1)}(b) always holds true (as a consequence of the first bound in \eqref{eq:Lm<Psi}).

\begin{twierdzenie}\label{th:main5} Let
$\nu(dx)=g(x)dx$ be a L\'evy measure such that $\nu(\Rdwithouto)=\infty$, $g(y)=g(-y)$ and {$g(x) \asymp f(|x|)$, $x \in \Rdwithouto$, for some nonincreasing function $f:\:(0,\infty)\to (0,\infty)$.} Moreover assume that $b \in \Rd$.

Then the following two conditions \textsc{(2.1)} and \textsc{(2.2)} are equivalent.
\begin{itemize}
\item[\textsc{(2.1)}] There exist $r_0>0$ and a constant $L_2 >0$ such that 
$$
\Psi(1/|x|) \leq L_2 |x|^d g(x), \quad |x| \leq 2r_0.
$$ 
\item[\textsc{(2.2)}] There exist $t_0, \theta, R>0$ such that $\theta h(t_0) \leq R$ and {constants} $C_1 - C_4$ such that for every $t \in (0,t_0]$ the transition densities $p_t$ exist and satisfy
\begin{align*}
  C_1 \, [h(t)]^{-d} \leq p_t(x+tb) \leq  C_3 \, [h(t)]^{-d}, \quad t \in (0,t_0], \quad |x| \leq \theta h(t),
\end{align*}
and
\begin{align*}
  C_2 \, t \, g(x) \leq p_t(x+tb) \leq  C_4 \, t \, g(x), \quad t \in (0,t_0], \quad \theta h(t) \leq |x| \leq R.
\end{align*}
\end{itemize}
\end{twierdzenie}
\noindent
{For a class of isotropic unimodal L\'evy processes (i.e. $g(x)$ and $p_t(x)$ are assumed to be strictly radially nonincreasing functions), it was recently proved in \cite{BGR13}} that the estimates $p_t(x) \asymp [h(t)]^{-d} \wedge t \Psi(1/|x|)|x|^{-d}$ for small $t$ and small $x$ are equivalent to the property that {the Matuszewska indices at infinity \cite[p. 68]{BGT} of the corresponding L\'evy-Khintchine exponent lies strictly between $0$ and $2$}. As shown there, in this class of processes, the latter property yields \textsc{(2.1)}. This reformulation of the condition \textsc{(2.1)} in terms of the L\'evy-Khintchine exponent easily extends to our settings (see Lemma \ref{lem:getE}). However, in general, the functions $g(x)$ and $p_t(x)$, corresponding to convolution semigroups investigated in the present paper, are not radially nonincreasing. Therefore, the proof of Theorem \ref{th:main5} requires more general methods than those {available} for isotropic unimodal case. 

Due to possible applications, it is {useful} to point out that our both conditions \textsc{(2.1)} and \textsc{(2.2)} in fact imply {two}-sided bound in the minimum form $p_t(x+tb) \asymp [h(t)]^{-d} \wedge t g(x)$ (see further discussion in Proposition \ref{prop:main5} and Remark \ref{rem:ex01}).

The characterization of \textsc{(1.2)} (in fact, the second bound in \textsc{(1.2)}) for large $x$ in terms of the convolution condition \textsc{(1.1)} (a) is given in Theorem \ref{th:main3} below. This result can be seen as the key and main ingredient of Theorem \ref{th:main1}. It was obtained independently of Theorem \ref{th:main5} under the following regularity condition \textbf{(E)} on $\Phi$ which is essentially more general than \textsc{(1.1)} (b).

\begin{itemize}
\item[\textbf{(E)}] There exist a constant $L_0>0$ and $t_p>0$ such that
\begin{align*}
  \int_{\Rd} e^{-t\Re\left(\Phi(\xi)\right)}|\xi|\, d\xi \leq L_0 \left[h(t)\right]^{-d-1},
  \quad t\in (0,t_p]. 
\end{align*}
\end{itemize}
The condition \textbf{(E)} {not only gives} the existence of densities $p_t\in C^1_b(\Rd)$ for $t\in (0,t_p]$, but it also provides {necessary} regularity of the small jump part of the process (see Preliminaries). Here we investigate the small time properties of the densities $p_t$ and it is intuitively clear that the study of the second bound in \textsc{(1.2)} for large $x$ {should also} require some regularity of the L\'evy-Khintchine exponent $\Phi$ for large arguments. One can verify that if there {exists} $\alpha >0$, $r_0>0$ and a constant $C \in (0,1]$ such that 
\begin{align*} 
\Psi(\lambda r) \geq C \lambda^{\alpha} \Psi(r), \quad \lambda \geq 1, \quad r>r_0,
\end{align*}
then \textbf{(E)} holds for all $t \in (0,1/\Psi(r_0))$ (see Lemma \ref{lem:getE}). On the other hand, it is clear that \textbf{(E)} excludes {symbols} that {vary} slowly (e.g. logarythmically) at infinity. In this case the integral on the left hand side is not finite for small $t$. 

Theorem \ref{th:main3} below gives the characterization of the sharp small time bounds of densities for big spatial arguments in terms of the decay of convolution of the L\'evy measures at infinity. 

\begin{twierdzenie}\label{th:main3} Let
$\nu(dx)=g(x)dx$ be a L\'evy measure such that $\nu(\Rdwithouto)=\infty$, $g(y)=g(-y)$ and {$g(x) \asymp f(|x|)$, $x \in \Rdwithouto$, for some nonincreasing function $f:\:(0,\infty)\to (0,\infty)$.} Moreover, assume that $b \in \R^d$ and that $\textbf{(E)}$ holds with some $t_p>0$. 

Then the following two conditions \textsc{(3.1)} and \textsc{(3.2)} are equivalent.
\begin{itemize}
\item[\textsc{(3.1)}] There exist $r_0>0$ and a constant $L_1>0$ such that 
\begin{equation*}
  g_{r_0} * g_{r_0}(x) \leq L_1 \, g(x), \quad |x| \geq 2r_0.
\end{equation*}
\item[\textsc{(3.2)}] There exist $t_0 \in (0,t_p]$, $R>0$, and {constants} $C_2, C_4$ such that we have
\begin{align*}
  C_2 \, t \, g(x) \leq 
  p_t(x+tb) 
  & \leq  C_4 \, t \, g(x), \quad t \in (0,t_0], \ |x| \geq R.
\end{align*}
\end{itemize}
In particular, if \textsc{(3.1)} is true for some $r_0$, then \textsc{(3.2)} holds with $R=4r_0$ and $t_0:= t_p \wedge \frac{1}{\Psi(1/r_0)}$. If \textsc{(3.2)} is true for some $t_0$ and $R$, then \textsc{(3.1)} holds for $r_0=R/2$. 
\end{twierdzenie}
\noindent
Note that Theorems \ref{th:main1} and \ref{th:main5} do not require {a priori to assume \textbf{(E)}}. Indeed, any of the equivalent conditions \textsc{(2.1)} and \textsc{(2.2)} (respectively \textsc{(1.1)} (b) and \textsc{(1.2)} for small $x$) {imply} that the L\'evy-Khintchine exponent $\Phi$ satisfies \textbf{(E)} (Lemma \ref{lem:getE}). In fact, the condition \textbf{(E)} is more general and covers an essentially larger class of semigroups than \textsc{(2.1)} (cf. Examples \ref{ex:ex0} and \ref{ex:ex2}(2)). In particular, the statement of Theorem \ref{th:main3} and the argument in its proof are completely independent of Theorem \ref{th:main5} and bounds \textsc{(2.1)}-\textsc{(2.2)}.

Theorem \ref{th:main3} determines when exactly the sharp two-sided bounds as in \eqref{eq:est2} (with $c_1 = c_2 =1$) are satisfied for large $x$. In particular, it shows that such bounds hold for a large class of symmetric tempered L\'evy processes for which they were not known before. Here the most interesting examples include processes {whose} jump intensities {are} exponentially and suboexponentially localized at infinity, even if the intensities of small jumps are remarkably different from {the} stable one, whenever the regularity condition \textbf{(E)} is satisfied (see Corollary \ref{cor:ex2} and Example \ref{ex:ex2}). Such tempered processes are important from the mathematical physics point of view (see e.g. \cite{CMS}) and, as we will see in the sequel, they seem to be quite interesting in the present context. For instance, if we consider the class of L\'evy processes with L\'evy measures $\nu(dx) = g(x)dx$ such that $g(x)=g(-x) \asymp |x|^{-\delta} e^{-m|x|}$, $\delta \geq 0$, $m>0$, for large $x$ (this covers important families of tempered L\'evy processes such as {the} relativistic stable or {the} Lamperti ones), then Proposition \ref{prop:convver} states that the convolution condition \textsc{(3.1)} holds true only exactly in two cases, when $\beta \in (0,1)$ and $\delta \geq 0$ or when $\beta =1$ and $\delta > (d+1)/2$. Theorem \ref{th:main3} (see also Corollary \ref{cor:ex2}) thus immediately settles that the two sided sharp bounds of the form \textsc{(3.2)} are satisfied for these two ranges of parameters only. In particular, they cannot hold when $\beta>1$ or when $\beta=1$ and $\delta \in (0,(d+1)/2)$. This {somewhat} surprising dichotomy property was not known before (see further discussion in Example \ref{ex:ex3}).

The study of the small time bounds in Theorem \ref{th:main5} and the lower bound in Theorem \ref{th:main3} is based on an application  of the results obtained recently in \cite{KSz1} and on some new tricky ideas. However, the most critical part of the paper is the proof of the upper bound in Theorem \ref{th:main3}. In fact, the primary motivation of our investigations was to understand and explain when exactly the upper bound as in \textsc{(3.2)} can be expected to hold, and how it can be described by the detail{ed} and direct properties of the corresponding L\'evy measure. The answer we give is that it is enough to know how fast the tail (or rather profile) of a single convolution of the Levy {measure} (restricted to the complement of some neighborhood of the origin) decays at infinity. This fits very well the convolutional structure of the semigroup $\left\{P_t: t \geq 0\right\}$, but it is a little unexpected that under the condition \textsc{(3.1)} the decay properties of all $n$-th convolutions of L\'evy measures appearing in the construction {are} decided exactly by the decay of the first one (we briefly recall the construction in Preliminaries). Note that the condition 
\textsc{(3.1)} has been recently discovered in \cite{KL} in a completely different context as a powerful tool to study the estimates of the eigenfunctions and some ultracontractivity properties of the Feynman-Kac semigroups for L\'evy processes. Moreover, similar convolution conditions, especially for the tails of measures, have been widely studied on the real line and the halfline in the context of various types of subexponentiality. It is known for many years that these properties play an important role in the study the relation between one-dimensional infinitely divisible distributions and their L\'evy measures (see e.g. \cite{W05, W10} and references therein). We would also like to mention that we {have} obtained recently in \cite{KSz} the upper bound for densities of Feller semigroups with jump kernels absolutely continuous with respect to the Lebesgue measure with densities that are dominated by some radial functions satisfying {a} condition as in \textsc{(3.1)}. However, the argument in this paper requires some additional smoothness of the majorizing functions and the additional regularity of the intensity of small jumps, which is in fact assumed to be of the stable type. 

The second important question we address in the present paper is {concerned with} the generality in which such type of convolution condition on L\'evy measure implies the sharp small time upper bound of the corresponding density $p_t(x)$ for large $x$, similar to \textsc{(3.2)}. We show that it holds true in {a} much more general case, extending far beyond the settings of Theorems \ref{th:main1} and \ref{th:main3}. Below we proceed in the general framework, {used also} in our recent paper \cite{KSz1}. We consider a large class of L\'evy measures satisfying the following localization (domination) condition from {the} above.

\begin{itemize}
\item[\textbf{(D)}] There exist a nonincreasing function $f:(0,\infty) \to (0,\infty)$, a parameter $\gamma \in [0,d]$, and a constant $L_3>0$ such that
\begin{align*}
 \nu(A) \leq L_3 f(\dist(A,0)) (\diam(A))^{\gamma},
\end{align*}
for every $A \in \Borel$ with $\dist(A,0)>0$. 
\end{itemize}
Here $\diam(A)$ is the diameter, $\dist(A,0)$ is the distance to $0$ of the set $A \subset \R^d$, and ${\mathcal{B}}(\Rd )$ denotes the Borel sets in $\Rd$. For comparison, in \cite{Knop13} the author considers a different type of a localization condition, which is based on the estimate of the tail of the L\'evy measure by the tail of some multidimensional (subexponential) distribution. Note that the condition \textbf{(D)} covers a large class of symmetric and asymmetric L\'evy measures which are not absolutely continuous with respect to {Lebesgue  measure}, including some product and discrete L\'evy measures and those with tails very fast decaying at infinity. 

Under the following convolution condition {\bf (C)}, naturally generalizing \textsc{(3.1)}, we obtain the sharpest possible upper bound of $p_t(x)$ for small $t$ and large $x$, which can be given by {making use of the} majorant $f$ satisfying the localization condition \textbf{(D)}.

\begin{itemize}
\item[\textbf{(C)}] There exist {constants} $L_1, L_4>0$ and $r_0>0$ such that for every $|x| \geq 2 r_0$ and $r \in (0,r_0]$
\begin{align*}
  \int_{|x-y|>r_0,\ |y|>r} f(|y-x|) \,\nu(dy) 
  \leq 
  L_1 \Psi(1/r) \ f(|x|) \quad \text{and} \quad f(r) \leq L_4 \Psi(1/r)r^{-\gamma},
\end{align*}
with $f$ and $\gamma$ given by the domination condition {\bf (D)}. 
\end{itemize}

\begin{twierdzenie}\label{th:main2} Let
$\nu$ be a L\'evy measure such that $\nu(\Rdwithouto)=\infty$ and let the assumptions \textbf{(E)}, \textbf{(D)} and \textbf{(C)} be satisfied for some $t_p>0$, the function $f$, the parameter $\gamma$ and some $r_0 >0$. Then there is a constant $C_5>0$ such that 
\begin{align*}
  p_t(x+tb_{h(t)})
  & \leq  C_5 \, t \, [h(t)]^{\gamma-d}\, f(|x|), \quad |x| > 4r_0, \ t \in (0,t_0],
\end{align*}
where $t_0:= t_p \wedge \frac{1}{\Psi(1/r_0)}$, and $b_r$ is given by \eqref{eq:def_br}.
\end{twierdzenie}
\noindent
The proof of Theorem \ref{th:main2} is critical for the whole paper. Its key argument are some sharp estimates of the $n$-th convolutions of restricted L\'evy measures (Lemma \ref{lm:conv_est3}) which are based on our new convolution condition {\bf (C)} and were not known before. 
One can check that under the assumptions of Theorem \ref{th:main3} both conditions {\bf (D)} and {\bf (C)} hold with $\gamma=d$ and $f$ being the profile of the density $g$. In this case, the convolution condition {\bf (C)} directly reduces to {the} assumption \textsc{(3.1)} (see Lemma \ref{lm:useful_2}). Theorem \ref{th:main3} is thus {a} direct corollary {of} Theorem \ref{th:main2}. Note that in light of the general property in \eqref{eq:Lm<Psi} below, the second inequality in {\bf (C)} is only a technical assumption saying that the profile $f$ is not too rough around zero.

We close the introduction by a brief discussion of the sharpness of our new convolution assumption in {\bf (C)} {as} compared {with} the condition \textbf{(P)} introduced recently in \cite{KSz1}, together with {\bf (D)} as a key assumption to study the upper bound for transition densities. 

\begin{itemize}
\item[\textbf{(P)}] 
There exists a constant $M>0$ such that 
\begin{align*}
 \int_{|y|>r} f\left(s\vee |y|-|y|/2 \right) \,\nu(dy) 
  \leq 
  M f(s) \Psi(1/r),\quad s>0,r>0,
\end{align*}
with $f$ given by \textbf{(D)}.
\end{itemize}
\noindent
Note that the structure of the condition \textbf{(P)} is much more isotropic than that of \textbf{(C)}, and therefore it is often more convenient to check. Under \textbf{(D)}, the condition \textbf{(P)} allowed us to get the result (see \cite[Theorem 1]{KSz1}) which, in particular, imply the upper bound as in Theorem \ref{th:main2}, but with rate $f(|x|/4)$ instead of $f(|x|)$, and with some additional exponentially-logarithmic correction term. At the {time when} the paper \cite{KSz1} {was written}, it was completely unclear whether the sharpest possible upper bound with $f(|x|)$ {could} be obtained under the condition \textbf{(P)}. In Proposition \ref{prop:rel}, although both conditions have completely different structure, we prove that \textbf{(C)} always implies the inequality in \textbf{(P)} for large $s$ and small $r$, but the converse implications is not true. This in fact means that the condition \textbf{(P)} is too weak to guarantee the optimal rate in the estimate of $p_t(x)$ for small $t$ and big $x$ in general. More precisely, it holds for a larger class than the convolution condition \textbf{(C)} and give some bounds for densities, but it cannot be used to derive the sharp bound as in Theorem \ref{th:main2} with the exact rate $f(|x|)$ imposed by the localization condition \textbf{(D)}. This is illustrated by Example \ref{ex:ex3}. 

The structure of the paper is as follows. In Preliminaries we collect all {the} facts needed in the sequel and briefly recall the construction of the semigroup $\left\{P_t:t \geq 0\right\}$. Based on that, we precisely explain what is the main object of our study in this paper. In Section \ref{sec:conv} we investigate the consequences of the condition $\textbf{(C)}$ and estimate the convolutions of L\'evy measures. In Section \ref{sec:proofs1} we prove Theorems \ref{th:main3} and \ref{th:main2} involving the bounds for $p_t(x)$ for large $x$. Section \ref{sec:proofs2} {is concerned with} small time bounds for small $x$. It includes the proof of Theorem \ref{th:main5}, the discussion of further implications, and formal proof of Theorem \ref{th:main1}. In Section \ref{sec:examples} we illustrate our results by various examples, including two less regular cases (Examples \ref{ex:ex1} and \ref{ex:ex11}), and {we} discuss the convolution condition with respect to some typical profiles of L\'evy densities (Proposition \ref{prop:convver} and Corollary \ref{cor:ex2}). In Subsection 6.3, we also illustrate the sharpness of our convolution condition $\textbf{(C)}$ {as} compared with $\textbf{(P)}$.


\section{Preliminaries}

We use $c,C,L$ (with subscripts) and $M$ to denote finite positive constants
which may depend only on $\nu$, $b$, and the dimension $d$. Any {\it additional} dependence
is explicitly indicated by writing, e.g., $c=c(n)$.
We write $f(x)\asymp g(x)$ whenever there is a constant $c$ such that $c^{-1}f(x) \leq g(x) \leq c f(x)$.

We will need the following preparation. As usual we divide the L\'evy measure in two parts. For $r>0$ we denote 
$$\mathring{\nu}_r(dy)=\indyk{B(0,r)}(y)\nu(dy) \quad \text{and} \quad \nubounded{r}(dy)= \indyk{B(0,r)^c}(y)\,\nu(dy).$$
In terms of the corresponding L\'evy process, $\mathring{\nu}_r$ is related to the jumps which are close to the origin, while $\nubounded{r}$ represents the large jumps.
Note that there exist constants $L_5, L_6$ such that for every $r>0$
\begin{equation}\label{eq:Lm<Psi}
  |\nubounded{r}| \leq L_5\, \Psi(1/r) \quad \text{and}\quad \Psi(2r) \leq L_6 \Psi(r),
\end{equation} 
which follows from \cite[Proposition 1]{KSz1} or \cite{Grz2013}. 

We now briefly recall the construction of the semigroup $\{P_t,\; t\geq 0\}$. For the restricted L\'evy measures we consider two semigroups of measures $\{\mathring{P}^r_t,\; t\geq 0\}$ and $\{\bar{P}^r_t,\; t\geq 0\}$ such that
$$
  \Fourier(\mathring{P}^r_t)(\xi) 
    =     \exp\left(t \int_{\Rdwithouto} \left(e^{i\scalp{\xi}{y}}-1-i\scalp{\xi}{y}\right)
            \mathring{\nu}_r(dy)\right)\, ,  \quad \xi\in\Rd\,,
$$
and
\begin{equation*}
  \Fourier(\bar{P}^r_t)(\xi) =
  \exp\left(t \int (e^{i\scalp{\xi}{y}}-1)\,
  \nubounded{r}(dy)\right)\, ,
  \quad \xi\in\Rd\, ,
\end{equation*}
respectively. We have
\begin{eqnarray*}
  |\Fourier(\mathring{P}^r_t)(\xi)| 
  &   =  & \exp\left(-t\int_{0<|y|<r}
           (1-\cos(\scalp{y}{\xi}))\,
           \nu(dy)\right) \nonumber \\
  &   =  & \exp\left(-t\left(\Re(\Phi(\xi))-\int_{|y|\geq r}
           (1-\cos(\scalp{y}{\xi}))\,
           \nu(dy)\right)\right) \nonumber \\
  & \leq & \exp(-t\Re(\Phi(\xi)))\exp(2t\nu(B(0,r)^c)),
           \quad \xi\in\Rd,
\end{eqnarray*}
and therefore by \textbf{(E)}, for every $r>0$ and $t\in (0,t_p]$ the measures $\mathring{P}^r_t$ are absolutely continuous with respect to
the Lebesgue measure with densities $\mathring{p}^r_t\in C^1_b(\Rd)$. 

We have 
\begin{align*} 
  P_t=\mathring{P}^r_t \ast \bar{P}^r_t \ast \delta_{t b_r}\,,\quad \ \text{and} \ \quad 
  p_t= \mathring{p}^r_t * \bar{P}^r_t\ast \delta_{t b_r} \,, \quad t>0,
\end{align*}
where $b_r$ is defined by (\ref{eq:def_br}), and 
\begin{eqnarray}\label{eq:exp}
  \bar{P}^r_t
  &  =  & \exp(t(\nubounded{r}-|\nubounded{r}|\delta_0)) =  \sum_{n=0}^\infty \frac{t^n\left(\nubounded{r}-|\nubounded{r}|\delta_0)\right)^{n*}}{n!} \\
  &  =  & e^{-t|\nubounded{r}|} \sum_{n=0}^\infty \frac{t^n\nubounded{r}^{n*}}{n!}\,,\quad t\geq 0\, . \nonumber
\end{eqnarray}
As usual, below we will use $\mathring{P}^r_t$, $\mathring{p}^r_t$ and $\bar{P}^r_t$ with $r=h(t)$ and
for simplification we will write $\mathring{P}_t=\mathring{P}^{h(t)}_t$, $\mathring{p}_t=\mathring{p}^{h(t)}_t$ and $\bar{P}_t=\bar{P}^{h(t)}_t$. As proven in \cite[Lemma 8]{KSz1}, if $\nu(\Rdwithouto)=\infty$ and \textbf{(E)} holds with $t_p>0$, then there exist constants $C_{6},C_{7}$ and $C_{8}$ such that
  \begin{equation*}
    \mathring{p}_t(x)\leq C_{6} \left[h(t) \right]^{-d} 
    \exp\left[  \frac{-C_{7}|x|}{h(t)}\log\left(1+\frac{C_{8}|x|}{h(t)}\right)\right],\quad t \in (0,t_p], \ x\in\Rd.
  \end{equation*}
Therefore, we always have 
\begin{align} \label{eq:basic_ineq}
p_t(x+tb_{h(t)}) = (\mathring{p}_t * \bar{P}_t)(x) \leq C_6 \left[h(t) \right]^{-d} \int_{\R^d} G(|y-x|/h(t)) \bar{P}_t(dy) ,\quad t \in (0,t_p], \ x\in\Rd,
\end{align}
with  
\begin{align*}
  G(s):=e^{-C_{7}s\log(1+C_{8}s)},\quad s\geq 0.
\end{align*}
\noindent
The main objective of the present paper is to find and study the precise estimates of convolutions $\nubounded{r}^{n*}$ and the measure $\bar{P}_t$ and, in consequence, also the optimal upper bound for the integral on the right hand side of \eqref{eq:basic_ineq} when $x$ is large. This will be achieved in the next two sections. 

\section{Convolutions of L\'evy measures} \label{sec:conv}

In this section we prove the sharp upper bounds for $n$-th convolutions of L\'evy measures, which are basic for our further investigations. 

First we discuss some decay properties of nonincreasing functions $f$ satisfying our new convolution condition \textbf{(C)}. They will be very important below.  

\begin{lemat}\label{lm:useful} 
Let $\nu$ be a L\'evy measure such that $\nu(\Rdwithouto)=\infty$ and let $f:(0,\infty) \to (0,\infty)$ be a nonincreasing function satisfying the first inequality in \textbf{(C)}. Then the following holds.
\begin{itemize}
\item[(a)] We have  
$$
f(s-r_0) \leq C_9 f(s), \quad s \geq 3r_0,
$$
with
\begin{align}
\label{eq:constC9}
C_9:= \inf\left\{c(x_0,\varepsilon) : 0 \neq x_0 \in \R^d, \ 0< \varepsilon<|x_0| \leq r_0/2 \right\} \geq 1, 
\end{align}
where
\begin{align*}
c(x_0,\varepsilon):= \left(\frac{L_1 \Psi(1/(|x_0|-\varepsilon))}{\nu(B(x_0,\varepsilon))}\right)^{\left\lceil  r_0/(|x_0|-\varepsilon)\right\rceil}.
\end{align*}
\item[(b)] There is a constant $C_{10}:=C_{10}(r_0)>0$ such that
$$
 G(s/(2r_0)) \leq C_{10} f(s), \quad s \geq r_0.
$$
\end{itemize} 
\end{lemat}

\begin{proof}
To prove (a) first observe that by the assumption $\nu(\Rdwithouto)=\infty$, there is $0 \neq x_0 \in \R^d$ and $0<\varepsilon$ such that $\varepsilon<|x_0| \leq r_0/2$ and $\nu(B(x_0,\varepsilon))>0$. Let $r_{\varepsilon}:= |x_0|-\varepsilon$. For $s \geq 2r_0$ let $\R^d \ni x_s:=(s/|x_0|)x_0$. For all $s \geq 2r_0$, by monotonicity of $f$ and \textbf{(C)}, we have 
\begin{align*}
L_1 \Psi(1/r_{\varepsilon}) f(s) & = L_1 \Psi(1/r_{\varepsilon}) f(|x_s|) \geq \int_{|x_s-y|>r_0,\ |y|>r_{\varepsilon}} f(|y-x_s|) \,\nu(dy) \\
& \geq \int_{B(x_0,\varepsilon)} f(|y-x_s|) \,\nu(dy) 
 \geq f(|x_s|-|x_0|+\varepsilon) \nu(B(x_0,\varepsilon)) \\ & = \nu(B(x_0,\varepsilon)) f(s-r_{\varepsilon}),
\end{align*}
which gives $f(s-r_{\varepsilon}) \leq (L_1 \Psi(1/r_{\varepsilon}))/\nu(B(x_0,\varepsilon)) f(s)$, for all $s \geq 2r_0$. The inequality in (a) follows from this with constant $C_9$ given by \eqref{eq:constC9}, for all $s \geq 3r_0$. Clearly, $C_9 \geq 1$, since $f$ is nonincreasing.

We now show (b). Let 
$$
n_{r_0}:= \inf \left\{ n \in \N: (1+C_8(n+2)/2)^{(C_7(n+2)/2)} > C_9^n/f(2r_0) \right\}.
$$
First we prove that the inequality 
$$
G((n+2)r_0/(2r_0)) = G((n+2)/2) \leq f((n+2)r_0), \quad n \geq n_{r_0},
$$
holds. If this is not true, then there is  $n \geq n_{r_0}$ such that 
\begin{align*}
f((n+2)r_0) & < G((n+2)/2) = e^{-(C_{7}(n+2)/2) \log(1+C_{8}(n+2)/2)} \\ & = (1+C_8(n+2)/2)^{-(C_7(n+2)/2)} < f(2r_0)/C_9^n.
\end{align*}
However, by (a) we have $0<f(2r_0) \leq C_9^n f((n+2)r_0)$, for every $n \in \N$. This gives a contradiction. We thus proved that the inequality in (b) holds with the constant $C_{10}=C_9$ for all $s \geq (n_{r_0}+2)r_0$ and therefore it also holds with $C_{10}:=C_9 \vee [f((n_{r_0}+2)r_0)]^{-1}$ for all $s \geq r_0$. 
\end{proof}

The following lemma yields the sharpest upper bound for the convolutions $\nubounded{r}^{\ast n}$ given by the profile function $f$ localizing the L\'evy measure from above in \textbf{(D)}. By \emph{sharpest bound} we mean here the estimate with the exact rate $f(\cdot)$ instead of $f(c \, \cdot)$ for some $c \in (0,1)$. Such bounds were not known before and it is a little bit surprising or even unexpected that the single estimate in \textbf{(D)} extends to all convolutions via the condition \textbf{(C)}. Weaker, not sharp, versions of Lemma \ref{lm:conv_est3} (b) with rates $f(c \, \dist(A,0))$ for some $c \in (0,1)$ were studied before (see e.g. \cite[Lemma 9]{KSz1}). However, our present result is based on a completely different argument using our new convolution condition \textbf{(C)}, which proved to be the optimal assumption to study such bounds. Lemma \ref{lm:conv_est3} will be a key argument in proving Theorem \ref{th:main2}.

\begin{lemat}
\label{lm:conv_est3}
Let
$\nu$ be a L\'evy measure such that $\nu(\Rdwithouto)=\infty$ and let the assumptions \textbf{(D)} and \textbf{(C)} be satisfied for some function $f$, the parameter $\gamma$ and some $r_0 >0$. 
Then the following hold.
\begin{itemize} 
\item[(a)] There is a constant $C_{11}=C_{11}(r_0)$ such that
\begin{align}
\label{eq:conv_est3}
\int_{|x-y|>r_0} f(|y-x|) \nubounded{r}^{n*}(dy) \leq \left(C_{11} \Psi\left(1/r \right)\right)^n f(|x|) ,\quad |x|\geq 3r_0, \ \ r \in (0,r_0], \ n \in \N.
\end{align}
\item[(b)] For every $ n \in \N$ and every bounded $A \in \Borel$ such that $\dist(A,0) \geq 3r_0-r_0/2^n$ we have
  \begin{equation}\label{eq:nu_n*_estnew}
    \nubounded{r}^{n*}(A) \leq C_{12}^n \left[\Psi(1/r)\right]^{n-1}f(\dist(A,0)) \ (\diam(A))^{\gamma}, \quad r \in (0,r_0],
  \end{equation}
  with a constant $C_{12}:=C_{12}(r_0, \left\lceil \diam(A)/r_0\right\rceil)$. 
\end{itemize}
\end{lemat}
   
\begin{proof}
First we consider (a). We prove that \eqref{eq:conv_est3} holds with $C_{11}$ given by \eqref{const:const}. For $n=1$ it is just the assumption \textbf{(C)}. 
Assume now that \eqref{eq:conv_est3} is true for some natural $n$ and all $x \in \Rd$ such that $|x|\geq 3r_0$. We will show that 
it holds also for $n+1$. For every $r \in (0,r_0]$ and $x \in \R^d$ with $|x|\geq 3r_0$ we have
\begin{align*}
\int_{|x-y|>r_0} f(|y-x|) \nubounded{r}^{(n+1)*}(dy) & = \int_{|x-z| < 3r_0 }\int_{|(x-z)-y|>r_0} f(|(x-z)-y|)\nubounded{r}^{n*}(dy) \nubounded{r}(dz)  \\ & \ \ \ + \int_{|x-z| \geq 3r_0} \int_{|(x-z)-y|>r_0} f(|(x-z)-y|)\nubounded{r}^{n*}(dy) \nubounded{r}(dz)  \\ & = I_1 + I_2.
\end{align*}
To estimate $I_1$ we consider two cases. When $3r_0\leq |x| < 5r_0$, then simply
$$
I_1 \leq f(r_0) |\nubounded{r}^{n*}| |\nubounded{r}|[f(5r_0)]^{-1} f(|x|) \leq f(r_0)L_5^{n+1} [f(5r_0)]^{-1} f(|x|) \left(\Psi\left(1/r \right)\right)^{n+1}.
$$ 
If now $|x| \geq 5r_0$, then by \textbf{(D)} and Lemma \ref{lm:useful} (a) we get
\begin{align*}
\nubounded{r}(B(x,3r_0)) & \leq L_3 \left[\Psi\left(1/r_0 \right)\right]^{-1} (3r_0)^{\gamma} \Psi\left(1/r \right) f(|x|-3r_0) \\ & \leq L_3 C_9^3 \left[\Psi\left(1/r_0 \right)\right]^{-1} (3r_0)^{\gamma} f(|x|) \Psi\left(1/r \right),
\end{align*} 
and, consequently, in this case,
$$
I_1 \leq f(r_0) |\nubounded{r}^{n*}| \nubounded{r}(B(x,3r_0)) \leq L_3 C_9^3 L_5^n \left[\Psi\left(1/r_0 \right)\right]^{-1} (6r_0)^{\gamma} f(r_0) f(|x|) \Psi\left(1/r \right)^{n+1}.
$$
To estimate $I_2$ it is enough to observe that by induction hypothesis we have 
$$
\int_{|(x-z)-y|>r_0} f(|(x-z)-y|)\nubounded{r}^{n*}(dy) \leq  \left(C_{11}\Psi\left(1/r \right)\right)^n f(|x-z|), \quad r \in (0,r_0],
$$
and, in consequence, by assumption \textbf{(C)}, 
$$
I_2 \leq \left(C_{11} \Psi\left(1/r \right)\right)^n \int_{|x-z| > r_0} f(|x-z|) \nubounded{r}(dz) \leq L_1 \left(C_{11}\right)^n  \left(\Psi\left(1/r \right)\right)^{n+1} f(|x|), 
$$
for every $r \in (0,r_0]$. Hence, \eqref{eq:conv_est3} holds for $n+1$, every $r \in (0,r_0]$ and every $x \in \R^d$ such that $|x| \geq 3r_0$ with constant 
\begin{align}
\label{const:const}
C_{11} =\left(L_5f(r_0)[f(5r_0)]^{-1} + L_3 C_9^3 f(r_0) \left[\Psi\left(1/r_0 \right)\right]^{-1} (6r_0)^{\gamma} + L_1\right) \vee L_5, 
\end{align}
and proof of (a) is complete. 

We now show (b). We prove the desired bound with constant $C_{12}$ given by \eqref{eq:const12}. When $n=1$ then our claim follows directly from \textbf{(D)}. Suppose now that \eqref{eq:nu_n*_estnew} is true for some $n \in \N$, all bounded sets $A \in \Borel$ such that $\dist(A,0) \geq 3r_0-r_0/2^n$ and every $r \in (0,r_0]$. We check \eqref{eq:nu_n*_estnew} for $n+1$. To shorten the notation let $\delta_A:=\dist(A,0)$. We consider two cases: $3r_0-r_0/2^{n+1} < \delta_A < 6r_0$ and $\delta_A \geq 6r_0$. Let first $3r_0-r_0/2^{n+1} < \delta_A < 6r_0$. Let
\begin{align*}
\nubounded{r}^{(n+1)*}(A) & = \int \nubounded{r}(A-y) \nubounded{r}^{n*}(dy) = \int_{|y|< \delta_A-r_0/2^{n+1}} + \int_{|y| \geq \delta_A-r_0/2^{n+1}} =: I_{11}+I_{12}.
\end{align*}
By \textbf{(D)}, the second part of \textbf{(C)} and \eqref{eq:Lm<Psi}, we have
\begin{align*}
I_{11} & \leq L_3 \int_{|y|< \delta_A-r_0/2^{n+1}} f(\delta_A-|y|) (\diam(A-y))^{\gamma}\nubounded{r}^{n*}(dy) \\ & \leq L_3  f(r_0/2^{n+1}) |\nubounded{r}^{n*}| (\diam(A))^{\gamma} \leq L_3L_4 (2^{n+1}/r_0)^{\gamma} \Psi(2^{n+1}/r_0) |\nubounded{r}^{n*}| (\diam(A))^{\gamma} \\ & \leq L_3L_4 [(r_0/2)^{\gamma} f(6r_0)]^{-1} \Psi(2/r_0) (2^{\gamma} L_5 L_6)^n f(\delta_A) (\Psi(1/r))^n (\diam(A))^{\gamma}.
\end{align*}
To estimate $I_{12}$ we just use the induction hypothesis and Lemma \ref{lm:useful} (a). Indeed, we have
\begin{align*}
I_{12} & = \int \nubounded{r}^{n*}((A-y) \cap B(0,\delta_A - r_0/2^{n+1})^c) \nubounded{r}(dy) \\ & \leq C_{12}^n f(\delta_A - r_0/2^{n+1})  (\Psi(1/r))^{n-1}(\diam(A))^{\gamma} |\nubounded{r}|  \leq L_5 C_9 C_{12}^n f(\delta_A)  (\Psi(1/r))^{n}(\diam(A))^{\gamma}.
\end{align*}

Let now $\delta_A > 6r_0$ and 
\begin{align*}
\nubounded{r}^{(n+1)*}(A) & = \int \nubounded{r}(A-y) \nubounded{r}^{n*}(dy) = \int_{|y| < \delta_A-3r_0}+ \int_{|y| \geq \delta_A-3r_0}  =: I_{21}+I_{22}.
\end{align*}
By exactly the same argument as for $I_{12}$, we get $I_{22} \leq L_5 C_9^3 C_{12}^n f(\delta_A)  (\Psi(1/r))^{n}(\diam(A))^{\gamma}$. It is enough to estimate $I_{21}$. By \textbf{(D)} and Lemma \ref{lm:useful} (a), we have
\begin{align*}
I_{21} & \leq L_3 (\diam(A))^{\gamma} \int_{|y| < \delta_A-3r_0} f(\dist(A,y)) \nubounded{r}^{n*}(dy) \\ & \leq L_3 (\diam(A))^{\gamma} C_9^{\left\lceil \diam(A)/r_0\right\rceil} \int_{|y-x_A|>r_0} f(|y-x_A|) \nubounded{r}^{n*}(dy),
\end{align*}
with some $x_A \in \R^d$ such that $|x_A|=\delta_A$. Thus, by \eqref{eq:conv_est3}, we conclude that 
$$
I_{21} \leq L_3  C_9^{\left\lceil \diam(A)/r_0\right\rceil}C_{11}^n f(\delta_A) (\Psi(1/r))^n (\diam(A))^{\gamma}
$$
and, therefore, \eqref{eq:nu_n*_estnew} holds with
\begin{align} \label{eq:const12}
C_{12}:= \max \left\{L_3, \ L_3L_4 [(r_0/2)^{\gamma} f(6r_0)]^{-1} \Psi(2/r_0) + L_5C_9^3 + L_3  C_9^{\left\lceil \diam(A)/r_0\right\rceil}, \ 2^{\gamma} L_5L_6, \ C_{11}\right\},
\end{align}
which completes the proof of the lemma.
\end{proof}

We now show that under the assumption that the L\'evy measure has a density which is comparable to some radially nonincreasing profile, the condition  \textsc{(3.1)} of Theorem \ref{th:main3} is in fact equivalent to {\bf (C)}.

\begin{lemat}\label{lm:useful_2} 
Let
$\nu(dx)=g(x)dx$ be a L\'evy measure such $\nu(\Rdwithouto)=\infty$ and let $f:\:(0,\infty)\to (0,\infty)$ be a nonincreasing function such that $g(x) \asymp f(|x|)$, $x \in \Rdwithouto$. Then the condition \textsc{(3.1)} of Theorem \ref{th:main3} is equivalent to {\bf (C)}.
\end{lemat}

\begin{proof}
Clearly, we only need to show that the condition \textsc{(3.1)} of Theorem \ref{th:main3} implies {\bf (C)}. For every $r \in (0,r_0)$ and $|x| \geq 2r_0$, by \textsc{(3.1)} and the monotonicity of $f$ and $\Psi$, we have
\begin{align*}
\int_{|x-y|>r_0,\ |y|>r} & f(|y-x|) g(y) dy \\ & \leq c \left(\int_{|x-y|>r_0,\ |y|>r_0} f(|y-x|) g(y) dy + \int_{r < |y| \leq r_0} f(|y-x|) \,g(y) dy \right)\\ 
& \leq c_1\left(\Psi(1/r_0) f(|x|)+ f(|x|-r_0) \nu(B(0,r)^c) \right)  \leq c_2 \Psi(1/r) (f(|x|)+f(|x|-r_0)).
\end{align*}
If now $|x| \geq 4r_0$, then by the comparability $g(x) \asymp f(|x|) >0$ and by similar argument as in Lemma \ref{lm:useful} (a) based on \textsc{(3.1)}, we get
\begin{align*}
f(|x|) & \geq c_3 \int_{B((2r_0/|x|)x,r_0)}f(|y-x|) f(|y|) dy \\ & \geq c_4 f(|x|-r_0) \int_{ B((2r_0/|x|)x,r_0)} f(|y|) dy \geq c_5 f(3r_0) r_0^d f(|x|-r_0) 
\end{align*}
and the first inequality in {\bf (C)} is satisfied. If $|x| \in [2r_0,4r_0]$, then by strict positivity and monotonicity of $f$, $f(|x|-r_0) \leq c_6 f(|x|)$, and the first bound in {\bf (C)} follows again. To show the second part of {\bf (C)} for $\gamma = d$ we observe that by \eqref{eq:Lm<Psi} we have
\begin{align} \label{eq:from4}
c_7 r^d f(r) \leq \int_{r/2 \leq |y| < r} g(y) dy \leq \nu(B(0,r/2)^c \cap B(0,r)) \leq L_5 \Psi(2/r) \leq L_5 L_6 \Psi(1/r), \quad r > 0. 
\end{align}
\end{proof}

\section{Proofs of Theorems \ref{th:main2} and \ref{th:main3} } \label{sec:proofs1}

We start with the following lemma which is a corollary from the estimates of the $n$-th convolutions of L\'evy measures proven in the previous section. 

\begin{lemat}
\label{lm:conv_est4}
Let
$\nu$ be a L\'evy measure such that $\nu(\Rdwithouto)=\infty$ and let the assumptions \textbf{(D)} and \textbf{(C)} be satisfied for some function $f$, the parameter $\gamma$ and some $r_0 >0$.  Recall that $\bar{P}_t=\bar{P}_t^{h(t)}$ and let $t_0:=\frac{1}{\Psi(1/r_0)}$. The following hold.
\begin{itemize} 
\item[(a)] We have 
\begin{align*}
\int_{|x-y|>r_0} f(|y-x|) \bar{P}_t(dy) \leq e^{C_{11}} f(|x|) ,\quad |x|\geq 3r_0, \ \ t \in (0,t_0].
\end{align*}
\item[(b)] For every bounded $A \in \Borel$ such that $\dist(A,0) \geq 3r_0$ we have
\begin{align*}
 \bar{P}_t(A) & \leq e^{C_{12}} t \ f(\dist(A,0))(\diam(A))^{\gamma}, \quad t \in (0,t_0]. 
\end{align*}
\end{itemize}
\end{lemat}

\begin{proof}
Statements (a) and (b) are direct consequences of \eqref{eq:exp} and estimates (a) and (b) in Lemma \ref{lm:conv_est3}, respectively.
\end{proof}

We are now ready to prove Theorem \ref{th:main2}. 

\begin{proof}[Proof of Theorem $\ref{th:main2}$]
By \eqref{eq:basic_ineq}, we only need to estimate the integral
$$
I:= \int_{\R^d} G(|y-x|/h(t)) \bar{P}_t(dy)
$$
for all $|x| \geq 4r_0$ and $t \in (0,t_0]$, where $t_0:=1/\Psi(1/r_0) \wedge t_p$ (recall that $t_p$ is given in \textbf{(E)}). 

By Lemma \ref{lm:useful} (b) (as a consequence of \textbf{(C)}), we have
\begin{align}
\label{cond:a3}
G(s/(2h(t_0))) \leq G(s/2r_0) \leq C_{10} f(s), \quad s \geq r_0.
\end{align}
Let now $t \in (0,t_0]$ and $|x| \geq 4r_0$. By \eqref{cond:a3}, we have
$$
G(|y-x|/h(t))\leq G(r_0/(2h(t)))G(|y-x|/(2h(t_0))) \leq C_{10} G(r_0/(2h(t))) f(|y-x|), \quad |y-x|>r_0, 
$$
and, consequently, we get
\begin{align*}
I & = \int_{|y-x|\leq r_0}  G(|x-y|/h(t)) \, \bar{P}_t(dy) + \int_{|y-x|> r_0} G(|x-y|/h(t)) \, \bar{P}_t(dy) \\
& \leq  \int_{|y-x|\leq r_0}  G(|x-y|/h(t)) \, \bar{P}_t(dy) +  C_{10} G(r_0/(2h(t))) \int_{|y-x|> r_0} f(|x-y|) \, \bar{P}_t(dy). 
\end{align*}
Denote both integrals above by $I_1$ and $I_2$, respectively. We first estimate $I_1$. By Fubini, we have
\begin{align*}
  I_1 & = \int_{|y-x|\leq r_0} \int_0^{G(|x-y|/h(t))} \, ds\, \bar{P}_t(dy) \\
  & = \int_0^1 \int \indyk{\{y\in\Rd:\: G(|x-y|/h(t))>s, |y-x|\leq r_0\}} \, \bar{P}_t(dy) ds \\
  & = \int_0^1 \bar{P}_t\left(B(x,r_0 \wedge h(t)G^{-1}(s))\right) ds \\
  & = \int_{G(r_0/h(t))}^1 \bar{P}_t\left(B(x,h(t)G^{-1}(s))\right) ds + G(r_0/h(t)) \bar{P}_t\left(B(x,r_0)\right).
\end{align*}
Applying now Lemma \ref{lm:conv_est4} (b) to both members above, we get
$$
I_1 \leq c e^{C_{12}} \left(t [h(t)]^{\gamma} f(|x|-r_0) \int_0^1  \left(G^{-1}(s)\right)^{\gamma}\, ds + t G(r_0/h(t)) f(|x|-r_0) r_0^{\gamma}\right),
$$
with $C_{12}=C_{12}(r_0,1)$ and finally, by using Lemma \ref{lm:useful} (a) and noting that $\int_0^1  \left(G^{-1}(s)\right)^{\gamma} ds <\infty$, we obtain  
$$
I_1 \leq c_1 t [h(t)]^{\gamma} f(|x|), \quad t \in (0,t_0],
$$
where $c_1=c_1(r_0)$. It is enough to estimate $I_2$ and $G(r_0/(2h(t)))$. We deduce directly from Lemma \ref{lm:conv_est4} (a) that 
$$
I_2 \leq e^{C_{11}} f(|x|), \quad t \in (0,t_0].
$$
Also, it follows from \cite[Lemma 3.6.22]{J1} that $\Psi(r) \leq 2 \Psi(1) (1+r^2)$, $r>0 $, and, in consequence,
$$
t[h(t)]^{\gamma} = \frac{[h(t)]^{\gamma}} {\Psi\left(\frac{1}{h(t)}\right)} \geq \frac{c_2 [h(t)]^{\gamma}}{1 + \frac{1}{[h(t)]^2}} = \frac{c_2 [h(t)]^{{\gamma}+2}}{1 + [h(t)]^2} \geq \frac{c_3}{1+r_0^2} [h(t)]^{{\gamma}+2} \geq c_4 G(r_0/(2h(t))),
$$
where $c_4=c_4(\Psi,r_0)$. Finally, we obtain 
$$
p_t(x+tb_{h(t)}) \leq C_6 [h(t)]^{-d} I \leq c_5 [h(t)]^{-d} \left(I_1 + G(r_0/(2h(t))) I_2\right) \leq c_6 t \ [h(t)]^{\gamma-d} \ f(|x|),
$$
with $c_6=c_6(\Psi,r_0)$, for $t \in (0,t_0]$ and $|x| \geq 4r_0$. This completes the proof. 
\end{proof}

\smallskip

\begin{proof}[Proof of Theorem $\ref{th:main3}$]
First note that under the assumption $g(x) \asymp f(|x|)$, $x \in \R^d$, the condition $\textbf{(D)}$ holds with $f$ and $\gamma = d$. Consider the implication \textsc{(3.1)} $\Rightarrow$ \textsc{(3.2)}. The upper bound in \textsc{(3.2)} is a direct consequence of Lemma \ref{lm:useful_2} and Theorem \ref{th:main2} with $t_0:= t_p \wedge 1/\Psi(1/r_0)$ and $R=4r_0$. The lower bound follows from \cite[Theorem 2]{KSz1} under the condition \textsc{(3.1)}. Indeed, by \cite[Theorem 2]{KSz1} we have
$$
p_{t}(x+tb) \geq c t f(|x|+c_1h(t_0)), \quad |x| \geq h(t_0), \ t \in (0,t_0].
$$
(the constant $C_6$ in the estimate (7) of \cite{KSz1} may be assumed to be smaller than $1$). By Lemma \ref{lm:useful} (a), we conclude that for every $|x| \geq 4r_0$ and $t \in (0,t_0]$ we have
$$
p_{t}(x+tb) \geq c_2 t g(x),
$$
which completes the proof of the first implication.

To prove the converse implication we assume that the estimates \textsc{(3.2)} hold. Let $r_0=R/2$. By the both bounds in \textsc{(3.2)} and by the semigroup property, we have
\begin{align*}
 g_{r_0} * g_{r_0}(x) & \leq c_3 \left(\frac{2}{t_0}\right)^2 \int_{|y-x|>r_0, \ |y| > r_0} p_{t_0/2}\left(x-y+\frac{t_0}{2}b\right) p_{t_0/2}\left(y+\frac{t_0}{2} b\right) dy \\ & \leq c_4 \left(\frac{2}{t_0}\right)^2 p_{t_0}(x+t_0b) \leq \frac{c_5}{t_0} g(x),
\end{align*}
for all $|x| \geq 2r_0$. The proof is complete.
\end{proof}

\section{Proofs of Theorems \ref{th:main5} and \ref{th:main1}, and related results} \label{sec:proofs2}

In Lemma \ref{lem:getE} below we collect some basic properties of the L\'evy-Khintchine exponents corresponding to the L\'evy measures investigated in Theorems \ref{th:main1}-\ref{th:main3}. In particular, we show that the condition \textsc{(2.1)} implies \textbf{(E)}. 
This will be used in the proofs of Theorem \ref{th:main5} and Proposition \ref{prop:main5} below.


\begin{lemat} \label{lem:getE}
Let
$\nu(dx)=g(x)dx$ be a L\'evy measure such that $\nu(\Rdwithouto)=\infty$, $g(y)=g(-y)$ and $g(x) \asymp f(|x|)$, $x \in \Rdwithouto$, for some nonincreasing function $f:\:(0,\infty)\to [0,\infty)$. Moreover, assume that $b \in \Rd$. Then the following hold.
\begin{itemize}
\item[(a)] There exists a constant $C_{13}$ such that 
$$
C_{13} \Psi(|\xi|) \leq \Re \Phi(\xi) \leq \Psi(|\xi|), \quad \xi \in \R^d \backslash \left\{0\right\}.
$$

\item[(b)] The condition \textsc{(2.1)} is equivalent to the property that there are $\alpha_1, \alpha_2 \in (0,2)$, $C_{14}, C_{15} >0$ and $s_0>0$ such that
\begin{align} \label{eq:scaling_old}
C_{14} \lambda^{\alpha_1} \Psi(s) \leq \Psi(\lambda s) \leq C_{15} \lambda^{\alpha_2} \Psi(s), \quad \lambda \geq 1, \quad s \geq s_0, 
\end{align}
that is, $\Psi$ has weak lower and upper scaling properties with indices $\alpha_1$ and $\alpha_2$ at infinity (see e.g. \cite[(17)-(18)]{BGR13}).
\item[(c)] If there are $\alpha >0$, $s_0>0$ and a constant $C_{16}>0$ such that 
\begin{align} \label{eq:wlsc}
\Psi(\lambda s) \geq C_{16} \lambda^{\alpha} \Psi(s), \quad \lambda \geq 1, \quad s \geq s_0,
\end{align}
then the condition \textbf{(E)} holds for all $t \in (0,1/\Psi(s_0))$.
\end{itemize}
\end{lemat} 

\begin{proof}
For $\xi \in \Rd$ we define $\Phi_0(\xi)=\Phi_0(|\xi|):= \int_{\Rd} (1-\cos(\scalp{\xi}{y})) f(|y|) dy$. It is known that there is an isotropic unimodal L\'evy process in $\Rd$ with the L\'evy-Khintchine exponent $\Phi_0$ and the L\'evy measure $\nu_0(dy)=f(|y|) dy$ (for the formal definition and further details on unimodal L\'evy processes we refer the reader to \cite{W83}). By comparability $g(x) \asymp f(|x|)$, $x \in \Rd$, we have
\begin{align} \label{eq:comp_Phis0}
\Re \Phi (\xi) \asymp \Phi_0(\xi), \quad \xi \in \R^d.
\end{align}
This and \cite[Proposition 1]{Grz2013} yields
\begin{align} \label{eq:comp_Phis}
\Psi(|\xi|) = \sup_{|z| \leq |\xi|} \Re \Phi(z) \asymp \sup_{|z| \leq |\xi|} \Phi_0 (z)=: \Psi_0(|\xi|) \asymp \Phi_0(\xi), \quad |\xi|>0.
\end{align}
The both properties \eqref{eq:comp_Phis0} and \eqref{eq:comp_Phis} give the assertion (a) of the lemma. 

We now prove (b). Suppose first that \textsc{(2.1)} holds. Then, by the inequality $\Psi(1/|x|) \leq L_2 |x|^d g(x)$, $0<|x| \leq 2r_0$, and by \eqref{eq:Lm<Psi} and the same argument as in \eqref{eq:from4}, we get
$$
\Psi_0 (1/r) \asymp r^d f(r), \quad r \in (0,2r_0],
$$
and finally, we derive from \cite[Theorem 26]{BGR13} that the function $\Psi_0$ has the property that there are $\alpha_1, \alpha_2 \in (0,2)$, $c_1, c_2 >0$ and $s_0>0$ such that
\begin{align} \label{eq:scaling}
c_1 \lambda^{\alpha_1} \Psi_0(s) \leq \Psi_0(\lambda s) \leq c_2 \lambda^{\alpha_2} \Psi_0(s), \quad \lambda \geq 1, \quad s \geq s_0.
\end{align}
By \eqref{eq:comp_Phis} also the function $\Psi$ has the scaling property as in \eqref{eq:scaling}. The converse implication in (b) uses exactly converse argument and it is omitted.

Since by (a) we have $\Re \Phi(\xi) \asymp \Psi(|\xi|)$ for $\xi \in \R^d \backslash \left\{0\right\}$, the property in assertion (c) can be established by following the argument (estimate) in \cite[Lemma 16]{BGR13}.
\end{proof}
\noindent


\smallskip

\begin{proof}[Proof of Theorem $\ref{th:main5}$]
We first consider the implication \textsc{(2.1)} $\Rightarrow$ \textsc{(2.2)}. Assume \textsc{(2.1)} and note that by Lemma \ref{lem:getE} the condition \textbf{\textbf{(E)}} is satisfied with some $t_p>0$. Moreover, observe that by \textsc{(2.1)}, \eqref{eq:Lm<Psi} and \eqref{eq:from4}, and the monotonicity of $f$, we have $f(r) \asymp \Psi(1/r)r^{-d}$ for $r \in (0,r_0]$ and, consequently, the doubling property $f(r) \asymp f(2r)$ holds for all $r \in (0,r_0/2]$. Thus, by \cite[Theorem 2]{KSz1} we obtain that there are $t_0 \in (0,t_p]$ and $\theta >0$ such that $\theta h(t_0) \leq R:=r_0/2$, for which the both lower bounds in \textsc{(2.2)} hold. To prove the upper bound define $f_0(r):=f(r) \vee f(r_0/2)$, $r>0$. Since $f(r) \leq f_0(r)$ for $r>0$, and $f_0(r)$ has a doubling property for all $r>0$, also the assumptions of \cite[Theorem 1]{KSz1} are satisfied with such profile function $f_0$ and we get
$$
p_t(x+tb) \leq c h(t)^{-d}\min\left\{1, t h(t)^{d} f_0(|x|) + e^{-c_1\frac{|x|}{h(t)} \log \left(1+ c_2 \frac{|x|}{h(t)}\right)} \right\}, \quad |x| >0, \quad t \in (0,t_p],
$$
with some constants $c, c_1, c_2>0$. In particular,
$$
p_t(x+tb) \leq c h(t)^{-d}\min\left\{1, t h(t)^{d} f(|x|) + e^{-c_1\frac{|x|}{h(t)} \log \left(1+ c_2 \frac{|x|}{h(t)}\right)} \right\}, \quad |x| \in (0,r_0/2], \quad t \in (0,t_p].
$$ 
It is enough to estimate the exponentially-logarithmic member in the above estimate. By \textsc{(2.1)}, for $|x| \in (0,r_0/2]$ and $t \in (0,t_p]$, we have
$$
t h(t)^{d} f(|x|) \geq L_2^{-1} \frac{h(t)^{d}}{|x|^d} \ \frac{\Psi\left(\frac{1}{|x|}\right)}{\Psi\left(\frac{1}{h(t)}\right)}.
$$
If $|x| \leq h(t)$, then we easily have $t h(t)^{d} f(|x|) \geq L_2^{-1} \geq L_2^{-1} e^{-c_1\frac{|x|}{h(t)} \log \left(1+ c_2 \frac{|x|}{h(t)}\right)}$. 
When $|x| > h(t)$, then by the doubling property of $\Psi$ in \eqref{eq:Lm<Psi}, we also get
$$
t h(t)^{d} f(|x|) \geq L_2^{-1} \frac{h(t)^{d}}{|x|^d} \ \frac{\Psi\left(\frac{1}{|x|}\right)}{\Psi\left(\frac{|x|}{h(t)} \frac{1}{|x|}\right)} \geq L_2^{-1} L_6^{-1} \left(\frac{|x|}{h(t)}\right)^{-d-\log_2 L_6} \geq c_3 e^{-c_1\frac{|x|}{h(t)} \log \left(1+ c_2 \frac{|x|}{h(t)}\right)}.
$$ 
In particular, we see that both upper bounds in \textsc{(2.2)} also hold for $R=r_0/2$, $t \in (0,t_0]$ and the same $\theta$.
 
We now show the opposite implication. Assume that \textsc{(2.2)} holds and let 
$$
r(t):=2 \pi \left(\int_{\R^d} e^{-t \Re \Phi(\xi)} d\xi\right)^{-\frac{1}{d}} = [p_t(tb)]^{-\frac{1}{d}}, \quad t>0.
$$ 
By the first two-sided bound in \textsc{(2.2)}, we have $c_4^{-1} h(t) \leq r(t) \leq c_4 h(t)$, $t \in (0,t_0]$, with some constant $c_4 \geq 1$. Since the function $r(t)$ is continuous in $(0,t_0]$, it is also onto the interval $(0,r(t_0)]$. Moreover, by the both bounds in \textsc{(2.2)}, for every $t \in (0,t_0]$, we have
$$
[h(t)]^{-d} \leq c_5 t f(\theta h(t)) = c_5 \frac{f(\theta h(t))}{\Psi(1/h(t))},
$$
and by the comparability of $r(t)$ and $h(t)$, we also get
\begin{align*}
[h(t)]^{-d} \geq (c_4^{-2} \theta)^d [c_4^{-1} \theta r(t)]^{-d}, \quad t \in (0,t_0],
\end{align*}
and
\begin{align*}
\frac{f(\theta h(t))}{\Psi(1/h(t))} \leq \frac{f(c_4^{-1} \theta r(t))}{\Psi((c_4^{-2}\theta)/(c_4^{-1}\theta r(t)))}, \quad t \in (0,t_0].
\end{align*}
Therefore 
$$
(c_4^{-2} \theta)^d [c_4^{-1} \theta r(t)]^{-d} \leq c_5 \frac{f(c_4^{-1} \theta r(t))}{\Psi((c_4^{-2}\theta)/(c_4^{-1}\theta r(t)))}, \quad t \in (0,t_0].
$$
Let now $r_0:= (2c_4)^{-1} \theta r(t_0)$ and note that for every $r \in (0, 2r_0]$ there is $t \in (0,t_0]$ such that $r=c_4^{-1} \theta r(t)$. We conclude that by doubling property and monotonicity of $\Psi$ it holds that
$$
\Psi(1/r)\leq c_6 f(r)r^d, \quad r \in (0,2r_0], 
$$
which is exactly \textsc{(2.1)}.
\end{proof}

\smallskip

The following proposition may be seen as the complement to Theorem \ref{th:main5}. One of its important consequences is that the bounds from \textsc{(2.2)} always imply two-sided bounds in the minimum form as in \textsc{(2.3)} below, while the converse implications holds true under the assumption \textbf{(E)}.

\begin{prop} \label{prop:main5}
Let
$\nu(dx)=g(x)dx$ be a L\'evy measure such that $\nu(\Rdwithouto)=\infty$, $g(y)=g(-y)$ and $g(x) \asymp f(|x|)$, $x \in \Rdwithouto$, for some nonincreasing function $f:\:(0,\infty)\to [0,\infty)$. Moreover, assume that $b \in \Rd$. Consider the additional conditions \textsc{(2.3)} and \textsc{(2.4)}

\begin{itemize}
\item[\textsc{(2.3)}] There exist $t_0 > 0$, $R>0$ and constants $C_{17}, C_{18}$ such that we have
\begin{align*}
  C_{17}   \min\left\{[h(t)]^{-d}, t g(x)\right\}\leq 
  p_t(x+tb) 
  & \leq  C_{18} \min\left\{ [h(t)]^{-d}, t g(x)\right\}, \quad t \in (0,t_0], \quad 0<|x| \leq R.
\end{align*}
\item[\textsc{(2.4)}] There exist $t_0 > 0$, $R>0$ and a constant $C_{19}$ such that we have
\begin{align*}
  p_t(x+tb) 
  & \leq  C_{19} \, t \, g(x), \quad t \in (0,t_0], \quad 0<|x| \leq R.
\end{align*}
\end{itemize}
Then the following hold. 
$$
\textsc{(2.1)} \ \Longleftrightarrow \ \textsc{(2.2)} \ \Longleftrightarrow \ \left[\bf{(E)} \ \mbox{and} \ \textsc{(2.3)} \right] \ \Longleftrightarrow \ \left[\bf{(E)} \ \mbox{and} \ \textsc{(2.4)} \right] 
$$
\end{prop}

\begin{proof}
Recall that by Theorem \ref{th:main5} the conditions \textsc{(2.1)} and \textsc{(2.2)} are equivalent, and by Lemma \ref{lem:getE} the condition \textsc{(2.1)} implies \textbf{(E)}. We also see that \textsc{(2.3)} implies \textsc{(2.4)} (in fact with no use of \textbf{(E)}). To complete the proof, we show the implications \textsc{(2.2)} $\Rightarrow$ \textsc{(2.3)}, [\textbf{(E)} \emph{and} \textsc{(2.3)}]  $\Rightarrow$ \textsc{(2.1)} and [\textbf{(E)} \emph{and} \textsc{(2.4)}]  $\Rightarrow$ \textsc{(2.1)}. Let $t_0:= \sup\left\{t \in (0,t_p]: h(t) \leq r_0\right\}$. By taking $r=\theta h(t)$ in \textsc{(2.1)} and \eqref{eq:Lm<Psi} for $t \in (0,t_0]$, by the same argument as in \eqref{eq:from4} and the doubling property of $\Psi$, we get $1/t = \Psi(1/h(t)) \asymp \Psi(1/(\theta h(t))) \asymp f(\theta h(t)) [\theta h(t)]^d$, $t \in (0,t_0]$. This clearly gives that there are constants $c_1, c_2 \geq 1$ such that
$$
[h(t)]^{-d} \leq c_1 t g(x), \quad |x| \leq \theta h(t), \ \ t \in (0,t_0], 
$$
and
$$
t g(x) \leq c_2 [h(t)]^{-d}, \quad |x| \geq \theta h(t), \ \ t \in (0,t_0]. 
$$
With these inequalities, two-sided bounds in \textsc{(2.3)} follow directly from \textsc{(2.2)}.

Proofs of implications [\textbf{(E)} \emph{and}  \textsc{(2.3)}]  $\Rightarrow$ \textsc{(2.1)} and [\textbf{(E)} \emph{and} \textsc{(2.4)}]  $\Rightarrow$ \textsc{(2.1)} are exactly the same. Indeed, \cite[Lemma 7]{KSz1} gives that there are $\theta, c_3 >0$ such that  
$$
p_t(x+tb) \geq c_3 [h(t)]^{-d}, \quad |x| \leq \theta h(t), \ \ t \in (0,t_p].
$$
By this estimate, the upper bound in \textsc{(2.3)} or \textsc{(2.4)} and the doubling property of $\Psi$ we thus get 
$$
[h(t)]^{-d} \leq c_4 t f(\theta h(t)) = c_4 \frac{f(\theta h(t))}{\Psi(1/h(t))} \asymp \frac{f(\theta h(t))}{\Psi(1/(\theta h(t)))}, \quad t \in (0,t_0].
$$
Now, by using a similar argument as in the second part of the proof of Theorem \ref{th:main5}, we obtain that there is $r_0, c_5>0$ such that 
$$
\Psi(1/r)\leq c_5 f(r)r^d, \quad r \in (0,2r_0], 
$$
and \textsc{(2.1)} holds. The proof is complete.
\end{proof}
\noindent
For further discussion of essentiality of the condition \textbf{(E)} in Proposition \ref{prop:main5} we refer the reader to Remark \ref{rem:ex01} in the last section. 

\smallskip

We close this section by giving the formal proof of Theorem \ref{th:main1}.

\begin{proof}[Proof of Theorem $\ref{th:main1}$]
First observe that by monotonicity and strict positivity of functions $f$ and $\Psi$ one can easily extend the estimate \textsc{(1.1)} (b) to $r \in (0,R]$ for every $R>2r_0$, possibly with worse constant $L_2$ dependent on $f(R)$ and $\Psi(1/R)$. In particular, it holds for $r \in (0,16r_0]$. Moreover, note that by Lemma \ref{lem:getE} the condition \textsc{(1.1)} (b) implies \textbf{(E)}. Therefore, the implication \textsc{(1.1)} $\Rightarrow$ \textsc{(1.2)} is a direct corollary from (proofs of) Theorems \ref{th:main5} and \ref{th:main3}. Indeed, \textsc{(1.2)} is a conjuction of \textsc{(2.2)} and \textsc{(3.2)} with the same $R=4r_0$. 

Consider now the implication \textsc{(1.2)} $\Rightarrow$ \textsc{(1.1)}. If both bounds in \textsc{(1.2)} hold, then the condition \textsc{(1.1)} (b) follows from Theorem \ref{th:main5} with some $r_0>0$. Therefore, as mentioned above, by Lemma \ref{lem:getE} also the assumption \textbf{(E)} is satisfied, and the condition 
\textsc{(1.1)} (a) can be directly derived from Theorem \ref{th:main3} with the same $r_0$. 
\end{proof}


\section{Further results, discussion and examples}\label{sec:examples}

In this section we discuss our results and some of their consequences in more detail. In particular, we illustrate them by several examples.

\subsection{Symmetric and absolutely continuous L\'evy measures}

We now illustrate the outcomes of our study with the L\'evy measures with densities that are comparable to some specific profile functions. Our Theorem \ref{th:main1} is a consequence of the two separate results for small and large $x$ given in Theorems \ref{th:main5} and \ref{th:main3}, respectively. Therefore, for more transparency, below we discuss these two theorems separately. 

As we mentioned in Introduction, the case of small $x$ is somewhat better explored. In the example below we test Theorem \ref{th:main5} on some L\'evy measures with the three different types of singularity at zero. Throughout this subsection we always assume that $f:(0,\infty) \to (0,\infty)$ is a nonincreasing profile function such that $f\cdot \indyk{(0,1]} = \kappa$ with $\kappa: (0,1] \to (0,\infty)$ 
such that $\lim_{r\downarrow 0} \kappa(r) = \infty$.

\begin{example} \label{ex:ex0} {\rm
\noindent
Consider the following three types of L\'evy measures $\nu(dx)=g(x)dx$ with $g(x) = g(-x) \asymp f(|x|)$, $x \in \Rdwithouto$, where the corresponding small jump profiles $\kappa$ are as follows. 
\begin{itemize}
\item[(1)] $\kappa(r) = r^{-d}$ (low intensity of small jumps) \\
or
\item[(2)] $\kappa(r) = r^{-d-\alpha_1} \log\left(1+1/r\right)^{\alpha_2}$ if $\alpha_1 \in (0,2)$ and $\alpha_2 \geq 0$ or, if $\alpha_1 \in (0,2)$ and $\alpha_2 < 0$, then $\kappa$ is a nonincreasing function such that $\kappa(r) \asymp r^{-d-\alpha_1} \log\left(1+1/r\right)^{\alpha_2}$ (note that for $\alpha_2<0$ the function $r^{-d-\alpha_1} \log\left(1+1/r\right)^{\alpha_2}$ need not be nonincreasing for all $r \in (0,1)$)\\
or
\item[(3)] $\kappa(r) = r^{-d-2} \log\left(1+1/r\right)^{-2}$ (high intensity of small jumps).
\end{itemize} 
By direct calculations based on \cite[Proposition 1]{KSz1}, one can show that for the above three cases we have:
\begin{itemize}
\item[(1)] $\Psi(r) \asymp \log(1+r)$, \ $h(t) \asymp e^{-1/t}$, 
\item[(2)] $\Psi(r) \asymp r^{\alpha_1} (\log(1+r))^{\alpha_2}$,  \ $h(t) \asymp t^\frac{1}{\alpha_1} \left[\log\left(1+\frac{1}{t}\right)\right]^{\frac{\alpha_2}{\alpha_1}}$,
\item[(3)] $\Psi(r) \asymp r^2 (\log(1+r))^{-1}$ (not $r^2 (\log(1+r))^{-2}$),  \ $h(t) \asymp t^\frac{1}{2} \left[\log\left(1+\frac{1}{t}\right)\right]^{\frac{-1}{2}}$,
\end{itemize} 
whenever $r \geq 1$ and $t \in (0,1/\Psi(1)]$. We thus see that the condition \textsc{(2.1)} is satisfied in the case (2) with $r_0=1/2$, but it fails for (1) and (3). Theorem \ref{th:main5} states that the small time sharp two sided bounds for small $x$ of the form \textsc{(2.1)} are true for convolution semigroups with L\'evy measures as in (2) only. In the remaining cases, they should be expected in a different form. Note that in some sense the singularities at zero as in (1) and (3) are borderline for the L\'evy measures that are required to satisfy $\nu(\left\{x:|x|<1\right\})=\infty$ and $\int_{0<|x|<1} |x|^2 \nu(dx) < \infty$. Informally speaking, this means that the condition \textsc{(2.1)} is typical for the measures having some balance between these two integrability properties (cf. \cite{BGR13}). For some available results on the estimates for transition densities corresponding to L\'evy processes with slowly varying characteristic exponent as in (1), we refer the reader to \cite[pages 117-118]{BBKRSV}. Sharp small time estimates for small $x$ for processes with high intensity of small jumps as in (3) are still an open and very interesting problem (see \cite{Mimica1, Mimica2} and recent discussion in \cite[page 22]{KSz1}). 
}
\end{example}  

The next remark is devoted to Proposition \ref{prop:main5}.

\begin{uwaga} \label{rem:ex01} {\rm
\noindent
It is reasonable to ask how far are bounds in \textsc{(2.3)} and \textsc{(2.4)} from \textsc{(2.2)} and \textsc{(2.1)} or, in other words, how essential for Proposition \ref{prop:main5} is the condition \textbf{(E)}. Consider the following observations.
\begin{itemize}
\item[(1)] First note that the condition \textbf{(E)} does not imply any of the conditions \textsc{(2.3)} and \textsc{(2.4)}. Counterexamples are here the L\'evy processes with high intensities of small jumps. For instance, if the L\'evy measure has the profile for small $x$ as in Example \ref{ex:ex0} (3), then \textbf{(E)} holds true, but the estimates in \textsc{(2.3)} and \textsc{(2.4)} fail.

\item[(2)] Also, when \textbf{(E)} does not hold, \textsc{(2.4)} does not imply any of the conditions \textsc{(2.2)} and \textsc{(2.1)}. Indeed, for example, the transition densities of the isotropic geometric $\alpha$-stable processes, $\alpha \in (0,2)$, in $\R^d$ fulfil the estimates (see \cite[Theorem 5.52]{BBKRSV})
$$
p_t(x) \asymp \frac{t}{|x|^{d-t\alpha}}, \quad |x| \leq 1, \ t \in (0, 1 \wedge d/(2\alpha)),
$$
which directly implies the upper bound in \textsc{(2.4)}. However, in this case, both conditions \textsc{(2.1)} and \textsc{(2.2)} fail. Since $\Phi(\xi) = \log(1+|\xi|^{\alpha})$, also the assumption \textbf{(E)} does not hold.  
\item[(3)] Since, in general, \textbf{(E)} does not imply \textsc{(2.3)}, one can ask if there are examples of processes with densities satisfying \textsc{(2.3)}, for which \textbf{(E)} and both conditions \textsc{(2.1)} and \textsc{(2.2)} are not true. However, it seems to be very difficult to indicate or construct such example. This problem remains open.
\end{itemize}  
}
\end{uwaga}

\smallskip

We now illustrate our Theorem \ref{th:main3}. To shorten the formulations below, first we set some useful notation. Recall that by $\kappa: (0,1] \to (0,\infty)$ we denote a nonincreasing function such that $\lim_{r\downarrow 0}\kappa(r)=\infty$. In the sequel we assume that
\begin{align} \label{eq:rangembetac}
m \geq 0, \quad \beta>0, \quad 0 < c \leq \kappa(1)e^{m}
\end{align}
and
\begin{align} \label{eq:rangedelta}
\delta \geq 0 \quad \text{when} \quad m > 0 \quad \text{and} \quad \delta > d \quad \text{when} \quad m=0.
\end{align}
Below we will consider the profiles $f:= f_{\kappa, m, \beta, \delta, c}$, where
\begin{align} \label{def:off}
f_{\kappa, m, \beta, \delta, c} (s) = \indyk{(0,1]}(s) \cdot \kappa(s)  + c \ \indyk{(1,\infty)}(s) \cdot e^{-m s^{\beta}} s^{-\delta} , \quad s > 0.
\end{align}
In general, a wider range of $\delta$ can be considered in \eqref{def:off}. However, we want $f_{\kappa, m, \beta, \delta, c}$ to be a strictly positive and nonincreasing profile function for the sufficiently regular density of the L\'evy measure. Therefore, in the remaining we always restrict our attention to the settings given by \eqref{eq:rangembetac}--\eqref{eq:rangedelta}.

For functions $f_{\kappa, m, \beta, \delta, c}$ we consider a convolution condition similar to \textsc{(2.1)}. 

\medskip

\begin{itemize}
\item[\textbf{(F)}]
There exists a constant $C=C(m,\beta,\delta)>0$ such that 
$$
\int_{|y-x|>1, \ |y|>1} f_{\kappa, m, \beta, \delta, c}(|y-x|)  f_{\kappa, m, \beta, \delta, c}(|y|) dy \leq C f_{\kappa, m, \beta, \delta, c}(|x|), \quad |x| \geq 2.
$$
\end{itemize}

\medskip

We will need the following proposition which gives the characterization of {\bf (F)} in terms of defining parameters $\beta$ and $\delta$. 

\begin{prop}\label{prop:convver}
Let \eqref{eq:rangembetac}--\eqref{eq:rangedelta} be satisfied. Then the condition {\bf (F)} holds if and only if 
\begin{itemize}
\item[(a)] $m=0$ and $\delta>d$  \\ or
\item[(b)] $m>0$, $\beta \in (0,1)$ and $\delta \geq 0$ \\ or
\item[(c)] $m>0$, $\beta =1$ and $\delta > (d+1)/2$. 
\end{itemize}
\end{prop}

\begin{proof}
We first prove that the given restrictions of parameters imply the condition {\bf (F)}. 
Since the case $m=0$ and $\delta>d$ is obvious, we only need to consider the case $m>0$, $\beta \in (0,1)$ and $\delta \geq 0$. To shorten the notation let $f:=f_{\kappa, m, \beta, \delta, c}$.
We start by justifying that for every $\beta \in (0,1)$ and $\eta \geq 0$ there is $s_0 \geq1$ such that 
\begin{align} \label{eq:eq20}
u^{\beta} + v^{\beta} \geq (u+v)^{\beta} + \eta  \log (u \wedge v),  \quad u, v \geq s_0.
\end{align}
This estimate is a consequence of the standard inequality 
$$(u \vee v)^{\beta} \geq (u+v)^{\beta} - \beta (u \wedge v)/(u \vee v)^{1-\beta}, \quad u, v >0,$$ 
and the fact that for any $\eta \geq 0$ we can find $s_0 \geq 1$ such that 
$$
(1-\beta) s^{\beta} \geq \eta \log s, \quad s \geq s_0.
$$
Indeed, with these inequalities, for every $u, v \geq s_0$, we get
\begin{align*}
u^{\beta} + v^{\beta}  = (u \vee v)^{\beta}+(u \wedge v)^{\beta}  & \geq (u+v)^{\beta} - \frac{\beta (u \wedge v)}{(u \vee v)^{1-\beta}}+ (1-\beta) (u \wedge v)^{\beta} + \beta (u \wedge v)^{\beta} \\
& \geq (u+v)^{\beta} + \eta \log (u \wedge v) + \frac{\beta (u \wedge v) }{(u \wedge v)^{1-\beta}}- \frac{\beta (u \wedge v)}{(u \vee v)^{1-\beta}} \\ & \geq (u+v)^{\beta} + \eta \log (u \wedge v), 
\end{align*}
which is exactly \eqref{eq:eq20}.

Let now $\eta := ((d+1-\delta)/m) \vee 0$ and find $s_0 \geq 1$ such that the inequality in \eqref{eq:eq20} holds (recall $\delta  \geq 0)$. When $s_0 >1$ and $|x| \in [2, 2s_0)$, then the inequality in {\bf (F)} holds by integrability and monotonicity properties of function $f$. Therefore, we consider only the case $|x| \geq 2s_0$. With this we have
\begin{align*}
\int_{|y-x|>1, \ |y|>1} & f(|y-x|) f(|y|) dy \leq 2 \left( \int_{1 < |y| \leq s_0} + \int_{s_0<|y| \leq |y-x|} \right) =: 2(I_1 + I_2 ).
\end{align*}
Since there is $c_1>0$ such that $f(s-s_0) \leq c_1 f(s)$ for every $s > 2s_0$, we get
$$
I_1 \leq f(|x|-s_0) \int_{|y|>1} f(|y|) dy \leq c_2 f(|x|).
$$ 
By the inequality \eqref{eq:eq20} with $\eta := ((d+1-\delta)/m) \vee 0$ applied to $u=|y-x|$ and $v=|y|$, we get
\begin{align*}
I_2  = & c_3 \int_{s_0<|y| \leq |y-x|} e^{-m[(|y-x|)^{\beta} + |y|^{\beta}]} (|y-x||y|)^{-\delta} dy \\
& \leq c_4 e^{-m|x|^{\beta}} |x|^{-\delta}\int_{1<|y| \leq |y-x|} |y|^{-\delta-m\eta} dy \leq c_5 e^{-m|x|^{\beta}} |x|^{-\delta} \int_{1<|y|} |y|^{-d-1} dy \leq c_6 f(|x|),
\end{align*}
which completes the proof of the first implication for $\beta \in (0,1)$ and $\delta \geq 0$.

We now consider the most interesting case when $m>0$, $\beta = 1$ and $\delta > (d+1)/2$. We will prove that with this range of parameters {\bf (F)} also holds true. When $d=1$, then this is an easy exercise. We consider only the case $d \geq 2$. Observe that the condition {\bf (F)} is in fact isotropic in the sense that it depends on the norm of $x$ only. Therefore we may and do assume that $x=(x_1,0,...,0)$ with $x_1 > 2$. Let
\begin{align*}   
I:= \int_{|y|>1, \ |y-x|>1} & f(|y-x|)  f(|y|) dy \\ & = \int_{|y|>1,\ y_1<1} + \int_{1 \leq y_1 \leq x_1-1} + \int_{|y-x|>1, \ y_1 > x_1-1} = I_1 + I_2 + I_3. 
\end{align*}
Note that both integrals $I_1$ and $I_3$ are bounded above by $f(|x|-1) \int_{|y|>1} f(|y|)dy \leq c_7 f(|x|)$. Thus, it suffices to estimate the integral $I_2$. Let $y=(y_1,...,y_d)$. By using spherical coordinates for $d-1$ integrals with respect to $dy_2...dy_d$, we get 
\begin{align*}   
I_2 & = c_8 \int_{1 \leq y_1 \leq x_1-1} e^{-m(|y-x|+|y|)} (|x-y||y|)^{-\delta} dy \\ & \leq c_9 |x|^{-\delta} \int_0^{\infty} \int_1^{x_1/2} e^{-m(\sqrt{(x_1-s)^2+r^2}+\sqrt{s^2+r^2})} \frac{r^{d-2}}{(s^2+r^2)^{\delta/2}} ds dr.
\end{align*}
One can directly check that for $(s,r) \in [1,x_1/2] \times [0,\infty)$ we have
\begin{eqnarray} 
\sqrt{(x_1-s)^2+r^2}+\sqrt{s^2+r^2} & = & x_1 + r^2 \left(\frac{1}{\sqrt{(x_1-s)^2 + r^2}+x_1-s} + \frac{1}{\sqrt{s^2 + r^2}+s}\right) \nonumber \\
& \geq & x_1 + \frac{r^2}{\sqrt{s^2 + r^2}+s} \geq x_1 + \frac{r^2}{\sqrt{s^2 + s r^2}+s} \label{eq:eqauxnice}
\end{eqnarray}
and, since $|x|=x_1$, in consequence, 
$$
I \leq c_{10} e^{-m|x|} |x|^{-\delta} \int_1^{x_1/2} \int_0^{\infty} e^{-\frac{m r^2}{\sqrt{s^2 + s r^2}+s}} \frac{r^{d-2}}{(s^2+r^2)^{\delta/2}} dr ds =: c_{11} e^{-m|x|} |x|^{-\delta} \ J(x_1).
$$
It is enough to prove that the function $J(x_1)$ given by the double integral above is bounded for all $x_1>2$. By using the substitution $r =\sqrt{s} u$, for every $x_1 >2$ we get 
$$
J(x_1) = \int_1^{x_1/2} \int_{0}^{\infty} e^{-\frac{m u^2}{\sqrt{1 + u^2}+1}} \frac{s^{\frac{d-1}{2}}u^{d-2}}{s^{\delta}(1+u^2/s)^{\delta/2}} du ds \leq \int_0^{\infty}  e^{-\frac{m u^2}{\sqrt{1 + u^2}+1}} u^{d-2} du \cdot \int_1^{\infty} s^{\frac{d-1}{2}-\delta} ds.
$$
We see that the first integral on the right hand side of the above inequality is always finite, while the second one is convergent whenever $\delta > (d+1)/2$. This completes the proof of the first implication.

To justify the opposite implication, we show that if $m>0$, $\beta>1$, $\delta \geq 0$ or if $m>0$, $\beta=1$ and $\delta \in [0,(d+1)/2]$, then the condition {\bf (F)} does not hold. Let $m>0$ and suppose first that $\beta>1$ and $\delta \geq 0$. For $|x| \geq 3$ we have
\begin{align*}
\int_{|y-x/2|<1} f(|y-x|) f(|y|) dy \geq c_{12} |x|^{-2\delta} e^{-m[(|x|/2+1)^{\beta}+(|x|/2+1)^{\beta}]} = c_{12} |x|^{-2\delta} e^{-2m(|x|/2+1)^{\beta}},
\end{align*}
and, consequently,
$$
\frac{\int_{|y-x|>1, \ |y|>1} f(|y-x|) f(|y|) dy}{f(|x|)} \geq c_{13} |x|^{-\delta} e^{m(|x|^{\beta} -2(|x|/2+1)^{\beta})} \to \infty, \quad \text{as} \quad |x| \to \infty,
$$
which shows that {\bf (F)} cannot hold. 

Suppose now that $\beta =1$ and $\delta \in [0, (d+1)/2]$. We will show that also in this case \textbf{(F)} does not hold. As before, we consider only the case $d \geq 2$. With no loss of generality we assume that $x=(2n,0,...,0)$, for natural $n \geq 2$. By using spherical coordinates for $d-1$ integrals with respect to $dy_2...dy_d$, where $y=(y_1,...,y_d)$, we get 
\begin{align*}   
I:=\int_{|y|>1, \ |y-x|>1} f(|y-x|) & f(|y|) dy \\ & \geq \int_{1<|y_1|<n} e^{-m(|y-x|+|y|)} (|x-y||y|)^{-\delta} dy \\ & \geq c_{14} |x|^{-\delta} \int_0^{\infty} \int_1^{n} e^{-m(\sqrt{(2n-s)^2+r^2}+\sqrt{s^2+r^2})} \frac{r^{d-2}}{(s^2+r^2)^{\delta/2}} ds dr.
\end{align*}
The same argument as in \eqref{eq:eqauxnice} yields that for all $(s,r) \in [1,n] \times [0,\infty)$ we have
\begin{align*}
\sqrt{(2n-s)^2+r^2}+\sqrt{s^2+r^2} & = 2n + r^2 \left(\frac{1}{\sqrt{(2n-s)^2 + r^2}+2n-s} + \frac{1}{\sqrt{s^2 + r^2}+s}\right) \\
& \leq 2n + \frac{2r^2}{\sqrt{s^2 + r^2}+s} \leq 2n + \frac{r^2}{s}
\end{align*}
and, in consequence,
$$
I \geq c_{14} e^{-m|x|} |x|^{-\delta} \int_1^{n} \int_0^{\infty} e^{-\frac{mr^2}{s}} \frac{r^{d-2}}{(s^2+r^2)^{\delta/2}} dr ds.
$$
Denote the last double integral by $I_n$. It is enough to show that $I_n \to \infty$ as $n \to \infty$. By using the substitution $r =\sqrt{s} u$, we get
$$
I_n =  \int_1^{n} \int_{0}^{\infty} e^{-mu^2} \frac{s^{\frac{d-1}{2}}u^{d-2}}{s^{\delta}(1+u^2/s)^{\delta/2}} du ds \geq \int_0^{\infty}  e^{-mu^2} \frac{u^{d-2}}{(1+u^2)^{\delta/2}} du \cdot \int_1^n s^{\frac{d-1}{2}-\delta} ds.
$$
Since $I_n \geq c_{15} \int_1^n s^{\frac{d-1}{2}-\delta} ds$, we see that $I_n \to \infty$ as $n \to \infty$ whenever $\delta \leq (d+1)/2$, which completes the proof of the proposition.
\end{proof}

The next corollary shows that for symmetric L\'evy processes with jump intensities comparable to radially nonincreasing functions $f_{\kappa, m, \beta, \delta, c}$ the restriction of parameters given by Proposition \ref{prop:convver} in fact characterizes two-sided bounds as in Theorem \ref{th:main3}. Note that for the class of convolution semigroups considered in Corollary \ref{cor:ex2} the regularity condition \textbf{(E)} depends only on the type of singularity of the function $\kappa$ at zero. 
 
\begin{wniosek} \label{cor:ex2} 
Assume that \eqref{eq:rangembetac}--\eqref{eq:rangedelta} hold. Let $\nu(dy)=g(y)dy$ be a L\'evy measure such that $\nu(\Rdwithouto)=\infty$ and $g(x) = g(-x) \asymp f_{\kappa, m, \beta, \delta, c} (|x|)$, $x \in \Rdwithouto$, and let $b \in \R^d$. Assume, moreover, that the assumption \textbf{(E)} is satisfied with some $t_p>0$. Then
\begin{align} \label{eq:ex2}
p_t(x+tb) \asymp t \ f_{\kappa, m, \beta, \delta, c} (|x|), \quad |x| \geq R, \quad t \in (0,t_0],
\end{align}
for some $R>0$ and $t_0 \leq t_p$ if and only if the one of the conditions (a), (b) or (c) in Proposition \ref{prop:convver} holds.  
\end{wniosek}

\begin{proof}
The result is a direct consequence of Theorem \ref{th:main3} and Proposition \ref{prop:convver}. Indeed, under the assumption of the theorem, the condition \textsc{(3.1)} in Theorem \ref{th:main3} is equivalent to \textbf{(F)}. 
\end{proof}
\noindent
The above result applies to a large class of pure jump symmetric L\'evy processes including the wide range of subordinate Brownian motions and more general unimodal L\'evy processes. It not only gives a sharp bound for the decay of the corresponding transition density at infinity for small time, but also settles when exactly this bound holds true. The most interesting examples are tempered L\'evy processes with jump intensities exponentially and suboexponentially localized at infinity. We cover a big subclass of tempered stable processes and others, even with more general intensities of small jumps that are remarkably different than the stable one, whenever the condition \textbf{(E)} is satisfied. 

Let us now briefly test our Theorem \ref{th:main3} with some exact examples of L\'evy processes with L\'evy measures absolutely continuous with respect to Lebesgue measure. 

\begin{example} \label{ex:ex2} {\rm
\noindent
\begin{itemize}
\item[(1)] One can see that in the case of the \emph{relativistic $\alpha$-stable process with parameter $\vartheta>0$} ($\alpha \in (0,2)$, $\kappa(r)=r^{-d-\alpha}$, $m=\vartheta^{1/\alpha}$, $\beta=1$, $\delta=(d+\alpha+1)/2$) \cite{ChKimSon} and the \emph{subexponentially and exponentially tempered $\alpha$-stable process} ($\alpha \in (0,2)$, $\kappa(r)=r^{-d-\alpha}$, $m>0$, $\beta \in (0,1]$, $\delta=d+\alpha$) \cite{Ros07} the condition\textsc{(3.1)} holds, and the corresponding densities satisfy two-sided small time sharp bounds as in \textsc{(3.2)}. 

\item[(2)] It is useful to see some other examples of $\kappa$ different from $r^{-d-\alpha}$ for which the background regularity assumption \textbf{(E)} in Theorem \ref{th:main3} (Corollary \ref{cor:ex2}) is satisfied. One can check that it still holds true when $\kappa$ is as in (2) and (3) of Example \ref{ex:ex0}. In particular, this shows that \textbf{(E)} covers a larger class of semigroups than \textsc{(2.1)}. On the other hand, as we mentioned in Introduction, \textbf{(E)} does not hold when the characteristic exponent $\Phi$ slowly varies at infinity. For instance, it fails for (1) in Example \ref{ex:ex0}.

\item[(3)] When $m>0$, $\beta=1$ and $\delta=0$ (e.g. \emph{Lamperti stable process} \cite{CPP10}), then the convolution condition \textsc{(3.1)} (or, simply, \textbf{(F)}) does not hold and Theorem \ref{th:main3} (Corollary \ref{cor:ex2}) states that the optimal bounds for large $x$ of the transition densities have to be of different form (cf. Example \ref{ex:ex3}).
\end{itemize}
}
\end{example}    

\bigskip

\subsection{More general L\'evy measures}

We now illustrate our Theorem \ref{th:main2} by discussing examples of (non-necessarily symmetric) L\'evy processes with more general L\'evy measures that are not absolutely continuous with respect to Lebesgue measure. 
 
\begin{example} \label{ex:ex1} \rm{\textbf{(Product L\'evy measures)} 
Let $m>0$, $\beta \in (0,1]$ and $\delta \geq 0$ and let
$$
\kappa(r) := r^{-1-\alpha_1} \left[\log\left(1+\frac{1}{r}\right) \right]^{-\alpha_2} \quad \text{for} \quad r \in (0,1],
$$
where the parameters $\alpha_1$ and $\alpha_2$ satisfy
\begin{itemize}
\item[(i)] $\alpha_1 \in (0,2)$, $\alpha_2 \in \left\{-2, 0 , 2\right\}$, \\
or 
\item[(ii)] $\alpha_1 = \alpha_2 = 2$.
\end{itemize}
Note that for such choice of $\alpha_2$ the function $\kappa$ is decreasing on the whole interval $(0,1)$.

\noindent
Let $\nu$ be a L\'evy measure such that 
\begin{align*} 
\nu(A) \asymp \int_{\sfera} \mu(d \theta) \left(\int_0^1 \indyk{A}(s \theta) \kappa(s) ds + c \int_1^{\infty} \indyk{A}(s \theta) s^{-\delta} e^{-ms^{\beta}} ds \right),
\end{align*}
where $c=\kappa(1)e^m$ and $\mu$ is a nondegenerate measure on $\sfera$ (sometimes called spectral measure) such that 
$$
\mu(\sfera \cap B(\theta,\rho)) \leq c_1 \rho^{\gamma-1}, \quad \theta \in \sfera, \ \rho>0,
$$
for some $\gamma \in [1,d]$. Let moreover $b \in \R^d$. 

Estimates of transition densities of convolution semigroups corresponding to such L\'evy measures have been studied in \cite{S10, S11, KSz1}, but the optimal small time bounds for large $x$ are still an open problem. Available results allow to get the upper bound for large $x$ with rate $|x|^{1-\gamma-\delta} e^{-m|c_2x|^{\beta}}$, for some constant $c_2 \in (0,1)$, and they cannot answer the question whether the correct rate is given by $|x|^{1-\gamma-\delta} e^{-m |x|^{\beta}}$. In particular, \cite[Theorem 1]{KSz1} yields
\begin{align*} 
p_t(x+tb_{h(t)}) \leq c_3 \, t \ [h(t)]^{\gamma-d} \, |x|^{1-\gamma-\delta} e^{-m\left|\frac{x}{4}\right|^{\beta}}, \quad t \in (0,1], \quad |x| \geq R >0,
\end{align*}
with
\begin{align} \label{eq:finalht_2}
h(t) \asymp \left\{
  \begin{array}{lcl}
    t^\frac{1}{\alpha_1} \left[\log\left(1+\frac{1}{t}\right)\right]^{\frac{-\alpha_2}{\alpha_1}} & \mbox{ for }  & \alpha_1 \in (0,2), \, \alpha_2 \in \left\{-2, 0 , 2\right\} \\
    t^\frac{1}{2} \left[\log\left(1+\frac{1}{t}\right)\right]^{\frac{-1}{2}} & \mbox{ for }  & \alpha_1 = \alpha_2 = 2 \\
  \end{array}
  \right.  ,  \quad t \in (0,1].
\end{align}
On the other hand, one can derive from \cite[Theorem 2]{KSz1} that if, furthermore, $\mu$ is symmetric and for some finite set $D_0=\{\theta_1,\theta_2,...,\theta_n\}\subset\sfera$, $n \in \N$, and the positive constants $c_4, \rho_0$ we have
$$
  \mu(\sfera \cap B(\theta,\rho)) \geq c_4 \rho^{\gamma-1},\quad \theta \in D_0, \, \rho \in (0,\rho_0],
$$
then there is $R \geq 1$ such that
 $$
          p_t(x+tb_{h(t)}) = p_t(x+tb) 
          \geq  c_5 \, t \ [h(t)]^{\gamma-d} \, |x|^{1-\gamma-\delta} e^{-m\left|x\right|^{\beta}}, \quad t \in (0,1], \, x \in D,
 $$
where $D=\{x\in\Rd:\: x=r\theta,\,r \geq R,\theta\in D_0\}$. Therefore, it is reasonable to ask what is the sharpest possible upper bound for the decay rate at infinity (cf. Example \ref{ex:ex3} below).

Our present Theorem \ref{th:main2} gives an answer to this problem. Indeed, whenever $\beta \in (0,1)$ and $\delta \geq 0$, or $\beta=1$ and, at least, $\delta > 1$, then it states that
$$
p_t(x+tb_{h(t)}) \leq c_6 \, t \ [h(t)]^{\gamma-d} \, |x|^{1-\gamma-\delta} e^{-m\left|x\right|^{\beta}}, \quad t \in (0,1], \quad |x| \geq 4.
$$

We now verify all assumptions of Theorem \ref{th:main2}. Denote $q(r):= \kappa(r) \indyk{0 < r \leq 1} + cr^{-\delta} e^{-mr^{\beta}}\indyk{r>1}$, and consider first the integral $\int_0^r s^2 q(s) ds$, $r>0$. For $r \in (0,1]$ we have
\begin{align*}
g(r)& := \int_0^r s^2 \kappa(s) ds  \asymp \left\{
  \begin{array}{lcl}
    r^{2-\alpha_1} \left[\log\left(1+\frac{1}{r}\right)\right]^{-\alpha_2} & \mbox{ for }  & \alpha_1 \in (0,2), \, \alpha_2 \in \left\{-2, 0 , 2\right\}, \\
    \left[\log\left(1+\frac{1}{r}\right)\right]^{-1} & \mbox{ for }  & \alpha_1 = \alpha_2 = 2, \\
  \end{array}
  \right.
\end{align*}
while for $r>1$ we get $g(r) = g(1) + \int_1^r s^{2-\delta} e^{-ms^{\beta}} ds \asymp c_7$. Thus, by \cite[Corollaries 2 and 3]{KSz1} 
\begin{align}\label{eq:finalrephi_2}
\Re(\Phi(\xi)) \asymp |\xi|^2 g(1/|\xi|), \quad \xi \in \R^d \backslash \left\{0\right\},
\end{align}
and, consequently, by checking that \eqref{eq:wlsc} holds, we can verify that the assumption \textbf{(E)} is satisfied for all $t \in (0,1]$. Moreover, by direct calculation using \eqref{eq:finalrephi_2} and the asymptotic formulas above, one can derive the asymptotics for $h(t)$, $t \in (0,1]$, as in \eqref{eq:finalht_2}. 

Also, it can be verified that for every $A \in \Borel$ with $\delta_A:=\dist(A,0)>0$ we have
$$
\nu(A) \leq c_8 \delta_A^{1-\gamma} \left(\kappa(\delta_A) \indyk{0<\delta_A \leq 1} + c \delta_A^{-\delta} e^{-m \delta_A^{\beta}} \indyk{\delta_A>1} \right) [\diam(A)]^{\gamma},
$$
which means that the assumption \textbf{(D)} holds with the function 
$$f(r):= \kappa(r) r^{1-\gamma}\indyk{0<r \leq 1} + c r^{1-\gamma-\delta} e^{-m r^{\beta}} \indyk{r>1}$$ and given $\gamma$.

By Lemma \ref{lm:verassCprod} below also the convolution condition in \textbf{(C)} for $r_0=1$ and such $f$ is satisfied when $\beta \in (0,1)$ and $\delta \geq 0$, or $\beta=1$ and $\delta >1$. The second part of \textbf{(C)} is an easy consequence of \eqref{eq:finalrephi_2}. Indeed, for $r \in (0,1]$ we have
$$
f(r) r^{\gamma} = r \kappa(r) \asymp r^{-\alpha_1} \left[\log\left(1+\frac{1}{r}\right) \right]^{-\alpha_2} = r^{-2} r^{2-\alpha_1} \left[\log\left(1+\frac{1}{r}\right) \right]^{-\alpha_2} \asymp \Psi(1/r)
$$
for $\alpha_1 \in (0,2)$ and
\begin{align*} 
f(r) r^{\gamma} = r \kappa(r) \asymp r^{-2} \left[\log\left(1+\frac{1}{r}\right) \right]^{-2} < r^{-2} \left[\log\left(1+\frac{1}{r}\right) \right]^{-1} \asymp \Psi(1/r)
\end{align*}
for $\alpha_1 = 2$. This completes the verification of assumptions of Theorem \ref{th:main2}. 
}
\end{example}

We now prove the auxiliary lemma which was needed in the previous example. 

\begin{lemat} \label{lm:verassCprod}
Let $m>0$, $\beta \in (0,1]$ and $\delta \geq 0$ and let $\nu$ be a L\'evy measure such that for every $A \in \Borel$ with $\dist(A,0)>1$
\begin{align*} 
\nu(A) \asymp \int_{\sfera} \int_1^{\infty} \indyk{A}(s \theta) s^{-\delta} e^{-ms^{\beta}} ds \mu(d \theta), 
\end{align*}
where $\mu$ is a measure on $\sfera$ such that there is $\gamma \in [1,d]$ for which 
$$
\mu(\sfera \cap B(\theta,\rho)) \leq c \rho^{\gamma-1}, \quad \theta \in \sfera, \ \rho>0,
$$
with some constant $c>0$. Let $f(r)= r^{1-\gamma-\delta} e^{-mr^{\beta}}$ for $r >1$. When $\beta \in (0,1)$ and $\delta \geq 0$ or $\beta =1$ and, at least, $\delta > 1$, then there is a constant $L_1>0$ such that we have
$$
\int_{|y-x|>1, \ |y|>r} f(|y-x|) \nu(dy) \leq L_1 f(|x|) \Psi(1/r), \quad r \in (0,1], \quad |x|>2.
$$

\end{lemat}

\begin{proof}
By \eqref{eq:Lm<Psi}, for $|x| \geq 2$ and $r \in (0,1]$, we have 
\begin{align*}
\int_{|y-x|>1, \ |y|>r} f(|y-x|) \nu(dy) & \leq c f(|x|-1) \Psi(1/r) + \int_{|y-x|>1, \ |y|>1} f(|y-x|) \nu(dy) \\ & \leq c_1 f(|x|) \Psi(1/r) + I,
\end{align*}
where
$$
I= \int_{\sfera} \int_1^{\infty} \indyk{B(x,1)^c}(s \theta)f(|s \theta - x|) f(|s\theta|) s^{\gamma-1} ds\mu(d \theta).
$$
Let first $\beta \in (0,1)$ and $\delta \geq 0$. Denote $\eta:= 2/m$ and find $s_0\geq1$ for which the inequality \eqref{eq:eq20} holds with such $\eta$. Note that for $|x| \in [2,2s_0]$ the desired inequality easily follows from properties of the function $f$ and it is enough to consider $|x|>2s_0$. Let
\begin{align*}
I & = \int_{\sfera} \left(\int_1^{s_0} + \int_{s_0}^{\infty}\indyk{B(x,s_0)\cap B(x,1)^c}(s \theta) + \int_{s_0}^{\infty}\indyk{B(x,s_0)^c}(s \theta)\right)f(|s \theta - x|) f(|s\theta|) s^{\gamma-1} ds\mu(d \theta) \\ &=:I_1+I_2+I_3.
\end{align*}
By the fact that $f(|x|-s_0) \leq c_2f(|x|)$ for $|x|>2s_0$, we get
$$
I_1 \leq c_2 f(|x|-s_0) \mu(\sfera) \int_1^{s_0} f(s) s^{\gamma-1} ds \leq c_3 f(|x|) \Psi(1/r), \quad r \in (0,1].
$$
Since $\int_{\sfera} \int_1^{\infty}\indyk{B(x,s_0)\cap B(x,1)^c}(s \theta) s^{\gamma-1} ds\mu(d \theta) \leq c_4 < \infty$ for every $|x| > 2s_0$, we also get
$$
I_2 \leq f(1) f(|x|-s_0) \int_{\sfera} \int_1^{\infty}\indyk{B(x,s_0)\cap B(x,1)^c}(s \theta) s^{\gamma-1} ds\mu(d \theta) \leq c_5 f(|x|) \Psi(1/r), \quad r \in (0,1].
$$
Thus, it is enough to estimate $I_3$. To this end, we use the inequality \eqref{eq:eq20} for $u=|x-s \theta|$ and $v=|s \theta|$. Similarly as in the first part of the proof of Proposition \ref{prop:convver}, we get
$$
I_3 \leq c_6 e^{-m|x|^{\beta}} \int_{\sfera} \int_{s_0}^{\infty} \indyk{B(x,s_0)^c}(s \theta) (|s \theta - x| \wedge |s \theta|)^{-2}(|s \theta - x||s \theta|)^{1-\gamma-\delta} s^{\gamma-1}  ds\mu(d \theta)
$$
and one can directly show that the last double integral is bounded by $c_7 |x|^{1-\gamma-\delta}$ for all $|x| \geq 2s_0$. Therefore, finally we obtain that $I_3 \leq c_8 f(|x|) \Psi(1/r)$ for all $r \in (0,1]$, which completes the proof of the lemma for $\beta \in (0,1)$. 

Let now $\beta=1$. In this case, we directly have  
$$
I \leq c_9 e^{-m|x|} \int_{\sfera} \int_{1}^{\infty} \indyk{B(x,1)^c}(s \theta) (|s \theta - x||s \theta|)^{1-\gamma-\delta} s^{\gamma-1}  ds\mu(d \theta)
$$
and when, at least, $\delta>1$, the last double integral again is bounded by $c_{10} |x|^{1-\gamma-\delta}$ for all $|x| \geq 2$. Therefore, again, the desired bound holds and the proof of the lemma is complete.
\end{proof} 

The next example is devoted to purely discrete L\'evy measures.

\begin{example} \label{ex:ex11} \rm{\textbf{(Discrete L\'evy measures)}
Let $\left\{v_n: n=1,..,k_0 \right\}$ be a family of $k_0 \in \N$, $k_0 \geq d$, vectors in $\R^d$ such that $\lin \left\{v_n: n=1,..,k_0 \right\} = \R^d$ and let $b \in \R^d$. For $q>0$ denote 
$$
A_q=\left\{x \in \R^d: x=2^{q n} v_k, \ \text{where} \ n \in \Z, \ k = 1, ..., k_0 \right\}
$$
and
$$
f(s):= \indyk{[0,1]}(s) \cdot s^{-\alpha/q}  +  e^m \indyk{(1,\infty)}(s) \cdot e^{-m s^{\beta}} s^{-\delta} , \quad s > 0, 
$$
with $m>0$, $\beta \in (0,1]$, $\delta >0$ and $\alpha \in (0,2q)$. Let 
$$
\nu(dy):= \int_{\R^d} f(|y|) \delta_{A_q}(dy) = \sum_{y \in A_q} f(|y|).  
$$
Clearly, $\nu$ is a not necessarily symmetric purely atomic L\'evy measure. By \cite[Proposition 1]{KSz1} we can easily check that in this case $\Re \Phi(\xi) \asymp |\xi|^2 \wedge |\xi|^{\alpha/q}$, $\xi \in \R^d \backslash \left\{0\right\}$, and $h(t) \asymp t^{q/\alpha}$, $t \in (0,1]$. Thus the assumption \textbf{(E)} is satisfied for all $t \in (0,1]$. Moreover, one can see that also the domination property \textbf{(D)} holds exactly with the function $f$ and $\gamma=0$. 

Some upper bound of transition densities in this case have been recently obtained: it follows from \cite[Theorem 1]{KSz1} that
$$
p_t(x+tb_{h(t)}) \leq c t^{-\frac{dq}{\alpha}} \left(1 \wedge t f(|x|/4)\right), \quad x \in \R^d, \quad t \in (0,1],
$$
On the other hand, under the additional assumption that $\nu$ is symmetric (i.e., for every $v \in \left\{v_n: n=1,..,k_0 \right\}$ we have $-v \in \left\{v_n: n=1,..,k_0 \right\}$), by \cite[Theorem 2]{KSz1} we also obtain
$$
p_t(x+tb_{h(t)}) = p_t(x+tb) \geq c_1 t^{-\frac{dq}{\alpha}} \left(1 \wedge t f(|x|)\right), \quad x \in A_q, \quad t \in (0,1].
$$
As in the previous example, arguments in \cite{KSz1} and other available results allow to get the upper bound for large $x$ with the rate $|x|^{-\delta} e^{-m|c_2x|^{\beta}}$, for some constant $c_2 \in (0,1)$, but not exactly $|x|^{-\delta} e^{-m |x|^{\beta}}$. Our present Theorem \ref{th:main2} says that in this case the optimal rate for large $x$ is indeed given by $|x|^{-\delta} e^{-m |x|^{\beta}}$ for all $m>0$, $\beta \in (0,1]$, $\delta >0$ and $\alpha \in (0,2q)$. In particular, we have 
$$
p_t(x+tb_{h(t)}) \leq c_3 t^{-\frac{dq}{\alpha}} \left(1 \wedge t f(|x|)\right), \quad x \in \R^d, \quad t \in (0,1].
$$
To justify this, it is enough to check the convolution condition in \textbf{(C)}. As before, it suffices to see that for some constant $c_4>0$
$$
\sum_{y \in A_q \cap B(0,1)^c \cap B(x,1)^c} f(|x-y|) f(|y|) \leq c_4 f(|x|), \quad |x| \geq 2.
$$
The left hand side is not larger than
\begin{align*}
e^{-m |x|^{\beta}} & \left(\sum_{y \in A_q \cap \left\{y:1 <|y| \leq |y-x| \right\}} (|x-y| |y|)^{-\delta} +  \sum_{y \in A_q \cap \left\{y:1 <|y-x| < |y| \right\}} (|x-y| |y|)^{-\delta}\right) \\
& \leq c_5 e^{-m |x|^{\beta}}|x|^{-\delta} \left(\sum_{y \in A_q \cap \left\{y:1 <|y| \leq |y-x| \right\}} |y|^{-\delta} +  \sum_{y \in A_q \cap \left\{y:1 <|y-x| < |y| \right\}} |x-y|^{-\delta}\right) \\ & \leq c_4 e^{-m |x|^{\beta}}|x|^{-\delta}, 
\end{align*}
which completes the justification.
}
\end{example} 
 
\bigskip

\subsection{Sharpness of the convolution condition \textbf{(C)}}

We now compare the convolution condition in \textbf{(C)} for small $r$ with the assumption \textbf{(P)} proposed recently in \cite{KSz1}. In Proposition \ref{prop:rel} and Example \ref{ex:ex3} below we show that \textbf{(C)} always implies the inequality in \textbf{(P)} for the same range of $r \in (0,r_0]$ and $s \geq 8r_0$, but there are L\'evy measures and corresponding functions $f$ satisfying \textbf{(P)}, for which \textbf{(C)} fails. Note that the following result does not require the condition \textbf{(D)}.  

\begin{prop}\label{prop:rel}
Let $\nu$ be a L\'evy measure and let $f:\:(0,\infty)\to (0,\infty)$ be a nonincreasing 
function such that there is $r_0>0$ and a constant $L_1$ satisfying
\begin{align}
\label{eq:strong}
\int_{|x-y|>r_0,\ |y|>r} f(|y-x|) \,\nu(dy) \leq L_1 \Psi(1/r) f(|x|), \quad |x| \geq 2 r_0, \quad r \in (0,r_0].
\end{align}
Then there exists an absolute constant $M>0$ such that
\begin{align}
\label{eq:weak}
  \int_{|y|>r} f\left(s\vee |y|-\frac{|y|}{2} \right) \,\nu(dy) 
  \leq 
  M \Psi(1/r) f(s), \quad s \geq 8 r_0, \quad r \in (0,r_0].
\end{align}
\end{prop}

\begin{proof}
We use the standard covering argument. First we introduce the two types of covers which will be used below. Let
\begin{align*}
n_0(d) := & \inf \Bigg\{n \in \N: \exists \ (x_k)_{k=1}^n \subset B(0,1)^c \ \text{s.t.} \ \forall \ 1\leq i\leq n \
\\ & 
\ \ \ \ \ \ \ \exists \ 1\leq j\leq n, \ 1/4 <|x_i-x_j|<3/8 \ \text{and} \ B(0,2) \cap B(0,1)^c \subset \bigcup_{k=1}^n B(x_k,1/2) \Bigg\}
\end{align*}
and
$$
n_1(d):=\inf\left\{n \in \N: \exists \ (z_k)_{k=1}^n \subset \sfera \ \text{s.t.}\ \sfera \subset \bigcup_{k=1}^n B(z_k,1/4) \right\}.
$$
Also, for $0 \neq z \in \R^d$ we denote
$$
\Gamma_z := \left\{y \in \R^d: \left|y - \frac{|y|}{|z|} z\right| \leq \frac{|y|}{4}\right\}.
$$
With this,  we may and do choose a finite sequence of points $(x_k)_{k=1}^{n_0} \subset B(0,1)^c$ such that for every $s>0$ we have $B(0,2s) \cap B(0,s)^c \subset \bigcup_{k=1}^{n_0} B(sx_k,s/2)$ and for every $x_i \in (x_k)_{k=1}^{n_0}$ there is another $x_j \in (x_k)_{k=1}^{n_0}$ such that $s/4 < |sx_i - sx_j|<3s/8$.

Similarly, we may and do find a sequence of points $(z_k)_{k=1}^{n_1} \subset \sfera$ such that for every $s>0$ we have $B(0,s) \subset \bigcup_{k=1}^{n_1} \Gamma_{s z_k}$. To verify the latter assertion it is enough to see that for every $0 \neq y \in B(0,s)$ there is $z_i \in (z_k)_{k=1}^{n_1}$ such that $(s/|y|)y \in B(sz_i,s/4)$. Thus $|y - (|y|/|sz_i|)sz_i| = (|y|/s)|(s/|y|)y-sz_i| \leq (|y|/s)(s/4)=|y|/4$ and, consequently, $y \in \Gamma_{sz_i}$.

We now apply the both covers above to estimate the integral on the left hand side of \eqref{eq:weak}. Let $r \in (0,r_0]$ and $s \geq 8r_0$ be fixed. We have
\begin{align*}
\int_{|y|>r} f\left(s\vee |y|-\frac{|y|}{2} \right) \,\nu(dy) & = \int_{r<|y|<4r_0} + \int_{4r_0 \leq |y|<s} + \int_{s\leq|y|<2s} + \int_{|y| \geq 2s} \\ & =: I_1(s)+I_2(s)+I_3(s)+I_4(s).
\end{align*}
Clearly, by Lemma \ref{lm:useful} (a) and \eqref{eq:Lm<Psi}, we have
$$
I_1(s) \leq f(s-2r_0)\nu(B(0,r)^c)  \leq c f(s) \Psi(1/r)
$$
and
$$
I_4(s) \leq f(s)\nu(B(0,2s)^c)  \leq \nu(B(0,r)^c) f(s) \leq c_1 f(s) \Psi(1/r).
$$
Hence, it is enough to estimate $I_2(s)$ and $I_3(s)$. We first consider $I_3(s)$. By using the first cover related to $n_0(d)$ introduced above, we have
\begin{align*}
I_3(s) 
& \leq \sum_{k=1}^{n_0} \left(\int_{|y| \geq s, \ s/8<|y-sx_k|<s/2}f\left(\frac{|y|}{2} \right) \,\nu(dy) \right. \\ & \ \ \ \ \ \ \ \ \ \ \ \ \ \ \ \ \ \ \ \ \ \ \ \ \ \ \ \ + \left.\int_{|y| \geq s, \ |y-sx_k|\leq s/8}f\left(\frac{|y|}{2} \right) \,\nu(dy) \right) 
\end{align*}
Recall that the above cover is chosen in the way that for every $x_i \in (x_k)_{k=1}^{n_0}$ we may find another $x_j \in (x_k)_{k=1}^{n_0}$ such that $\overline B(sx_i,s/8) \subset B(sx_j,s/2) \backslash \overline B(sx_j,s/8)$. By this fact we may be sure that
$$
I_3(s) 
 \leq  \sum_{k=1}^{n_0} n_k \int_{|y| \geq s, \ s/8<|y-sx_k|<r/2}f\left(\frac{|y|}{2}\right)  \,\nu(dy), 
$$
where $n_k-1 \geq 0$ is a number of small balls $B(sx_i,s/8)$ covered by $B(sx_k,s/2) \backslash \overline B(sx_k,s/8)$ for given $k$ (we assume that each ball $B(sx_i,s/8)$ is covered once). Clearly, $\sum_{k=1}^{n_0}n_k=2n_0$. Now since for $y \in \overline B(0,s)^c \cap B(sx_k,s/2)$ we have $|y-sx_k|<s/2 \leq |y|/2$, we get
\begin{align*}
I_3(s) & \leq  \sum_{k=1}^{n_0} n_k \int_{|y| \geq s, \ s/8<|y-sx_k|<s/2}f(|y-sx_k|) \,\nu(dy) \\ & \leq  \sum_{k=1}^{n_0} n_k \int_{|y| > r_0, \ r_0<|y-sx_k|}f(|y-sx_k|) \,\nu(dy)
\end{align*}
and by \eqref{eq:strong} and monotonicity of $\Psi$, we finally obtain
\begin{align*}
I_3(s) & \leq L_1 \Psi(1/r_0) \sum_{k=1}^{n_0} n_k  f(|sx_k|) \leq c_2 f(s) \Psi(1/r).
\end{align*}
It remains to estimate $I_2(s)$. This will be done by using the second cover related to $n_1(d)$. With this we have
\begin{align*}
I_2(s) & \leq \sum_{k=1}^{n_1} \int_{\left\{y:\, 4r_0<|y|<s\right\} \cap \Gamma_{sz_k}}f(s-|y|/2) \,\nu(dy).
\end{align*}
Whenever $|y|>4r_0$ and $y \in \Gamma_{sz_k}$ for some $k\in \left\{1,...,n_1\right\}$, we have 
\begin{align*}
s-|y|/2 & = s- (3/4)|y| + |y|/4 > (4r_0)/4 + (s-|y|)  + |y|/4 \\ 
& \geq |(s+r_0)z_k - sz_k| + |sz_k - (|y|/s)(sz_k)| + |(|y|/s)(sz_k) - y| \\ & \geq |(s+r_0)z_k - y|
\end{align*}
and, finally, 
\begin{align*}
I_2(s) & \leq \sum_{k=1}^{n_1} \int_{\left\{y:4r_0<|y|<s\right\} \cap \Gamma_{sz_k}}f(|(s+r_0)z_k - y|) \,\nu(dy) \\ & \leq \sum_{k=1}^{n_1} \int_{|y|> r_0, \  |(s+r_0)z_k - y| > r_0 }f(|(s+r_0)z_k - y|) \,\nu(dy).
\end{align*}
One more use of \eqref{eq:strong} and monotonicity of $\Psi$ gives
$$
I_2(s) \leq L_1 \sum_{k=1}^{n_1}f(|(s+r_0)z_k|) \Psi(1/r_0) \leq c_3 f(s) \Psi(1/r).
$$
The proof is complete.

\end{proof}

We now give some examples for which the converse implication in Proposition \ref{prop:rel} does not hold.

\begin{example} \label{ex:ex3} {\rm
Let $\nu(dy)= g(y)dy$ be a L\'evy measure such that $g(y) = g(-y) \asymp |y|^{-d-\alpha}(1+|y|)^{d+\alpha-\delta} e^{-m|x|^{\beta}}$ for $y \in \R^d \backslash \left\{0\right\}$, where $m>0$, $\beta \in (0,1]$, $\alpha \in (0,2)$, $\delta \geq 0$ (for simplicity we assume here that $b=0$). One can directly check that for this range of parameters the condition \textbf{(P)} is always satisfied for all $r >0$ and $s>0$ while, as proved in Proposition \ref{prop:convver}, the condition \textbf{(C)} does not hold when $\beta=1$ and $\delta \leq (d+1)/2$. By using \cite[Theorems 1-2]{KSz1} (cf. \cite{Knop13, ChKimKum}) we obtain that for all  $m>0$, $\beta \in (0,1]$, $\alpha \in (0,2)$ and $\delta \geq 0$ it holds that
$$
c_1 \left( t^{-d/\alpha} \wedge t g(x) \right) \leq p_t(x) \leq c_2 \left( t^{-d/\alpha} \wedge t g(x/4) \right), \quad x \in \R^d, \quad t \in (0,1]. 
$$
Similarly as in the examples discussed in the previous sebsection, this bound may suggest that the correct decay rate of $p_t(x)$ at infinity is again given exactly by $g(x)$. However, our Theorem \ref{th:main3} (Corollary \ref{cor:ex2}) states that when $\beta=1$ and $\delta \leq (d+1)/2$ then $g(x/4)$ in the upper bound for large $x$ cannot be replaced by $c_3 g(x)$ for any $c_3 >0$. This means that in this case the lower estimate is too weak and the correct two-sided bound for densities $p_t(x)$ is of different form!

The case $\beta=1$ and $\delta=(d+1)/2$ determines some kind of \textit{'phase transition'} in the dynamics of the convolution semigroups $(P_t)$. This subtle dichotomy phenomenon cannot be seen from the previously known results on the asymptotic behaviour of the kernels of jump processes (see e.g. \cite{S10, S11, ChKum08, ChKimKum, KSch1, KSz1, Knop13}), but also it cannot be definitively explained at this stage of our study. We close the discussion by recalling that this range of parameters also involves some well known and important classes of tempered L\'evy processes such as Lamperti stable ones and others, for which the optimal bounds of densities are still an open problem. 
}
\end{example}
{
\textbf{Acknowledgements.} We thank the anonymous referee for comments and suggestions on the paper. We also thank K. Pietruska-Pa{\l}uba for her reading of the preliminary version of the manuscript and useful suggestions. }

\end{document}